\documentclass[11pt,letterpaper]{article}

\usepackage{fullpage,verbatim} 
\usepackage{amsmath,amsthm,amssymb,amsfonts,geometry}  
\usepackage{color,setspace}
\usepackage{tikz} 
\usetikzlibrary{arrows,shapes,fit} 
\usepackage{booktabs} 
\usepackage{comment} 
\usepackage{url}
\usepackage{authblk}
\usepackage{eurosym}

\usepackage{algorithm}

\newcommand{\keywords}[1]{\par\addvspace\baselineskip \noindent Keywords: \enspace\ignorespaces#1}

\newtheorem{theorem}{Theorem}[section]

\newtheorem{prop}[theorem]{Proposition}

\newtheorem{assumption}[theorem]{Assumption}

\newcommand{\BI}{\begin{itemize}} \newcommand{\EI}{\end{itemize}}
\newcommand {\BE}{\begin{enumerate}} \newcommand {\EE}{\end{enumerate}}
\newcommand{\I}{\item}
\newcommand {\bena}{\begin{enumerate}\renewcommand{\labelenumi}{\alph{enumi}.}\item}

\newcommand {\beqn}{\begin{equation}}\newcommand {\eeqn}{\end{equation}}
\newcommand {\beqan}{\begin{eqnarray}}\newcommand {\eeqan}{\end{eqnarray}}
\newcommand {\beqa}{\begin{eqnarray*}}\newcommand {\eeqa}{\end{eqnarray*}}
\newcommand {\barr}{\begin{array}}\newcommand {\earr}{\end{array}}
\newcommand {\bat}{\begin{tabular}}\newcommand {\eat}{\end{tabular}}
\newcommand {\bsubeq}{\begin{subequations}}\newcommand {\esubeq}{\end{subequations}}

\newcommand{\projx}{\textup{proj}}

\newcommand{\exclude}[1]{}

\newcommand{\basetime}{$t_{\mathrm{LDR}}$}

\newcommand{\cred}{\color{black}}

\newcommand{\Exp}{\mathbb{E}}
\newcommand{\R}{\mathbb{R}}

\newcommand{\ra}{\rightarrow}

\newcommand{\cpar}{c}
\newcommand{\hpar}{h}
\newcommand{\Apar}{A}
\newcommand{\Bpar}{B}
\newcommand{\Cpar}{C}
\newcommand{\bpar}{b}
\newcommand{\Tpar}{T}
\newcommand{\Gpar}{E}

\newcommand{\Dpar}{D}
\newcommand{\dpar}{d}
\newcommand{\Ffnc}{{\cred \mathcal{Q}}}
\newcommand{\ffnc}{{\cred Q}}

\newcommand{\rpar}{r}
\newcommand{\kpar}{\ell}  
\newcommand{\ksubt}{\kpar_t}
\newcommand{\ksubs}{\kpar_\rpar}

\newcommand{\kt}{\kpar^t}

\newcommand{\Kpar}{K}
\newcommand{\Kt}{\Kpar_t}
\newcommand{\Ktmone}{\Kpar_{t-1}}
\newcommand{\Ktpone}{\Kpar_{t+1}}
\newcommand{\ppar}{p}
\newcommand{\qpar}{q}
\newcommand{\pt}{\ppar_t}
\newcommand{\qt}{\qpar_t}
\newcommand{\mpar}{m}
\newcommand{\mt}{\mpar_t}

\newcommand{\npar}{n}
\newcommand{\nt}{\npar_t}
\newcommand{\Ipar}{I}
\newcommand{\Jpar}{J}

\newcommand{\cumxipar}{\zeta}

\newcommand{\xipar}{\xi}
\newcommand{\xiT}{\xipar^T}
\newcommand{\xisubt}{\xipar_t}

\newcommand{\xit}{\xipar^t}

\newcommand{\xitmone}{\xipar^{t-1}}
\newcommand{\xitpone}{\xipar^{t+1}}
\newcommand{\allxi}{\mathbb{P}\text{-a.s.}}
\newcommand{\xisupp}{\Xi}
\newcommand{\Xset}{X}

\newcommand{\basis}{\mathrm{\Phi}}
\newcommand{\xvar}{x}
\newcommand{\hxvar}{\hat{x}}
\newcommand{\laxvar}{x^{STT}}
\newcommand{\svar}{s}
\newcommand{\lasvar}{s^{STT}}

\newcommand{\betavar}{\mathrm{\beta}}
\newcommand{\hbetavar}{\hat\betavar}
\newcommand{\betat}{\betavar_t}

\newcommand{\thetavar}{\mathrm{\Theta}}
\newcommand{\thetat}{{\cred \thetavar_{t}}}

\newcommand{\Lvar}{\lambda}
\newcommand{\Gvar}{\gamma}
\newcommand{\Lvart}{\Lvar_t(\xit)}
\newcommand{\Gvart}{\Gvar_t(\xit)}
\newcommand{\alphavar}{\mathrm{\Lambda}}
\newcommand{\halphavar}{\hat{\mathrm{\Lambda}}}
\newcommand{\hGvar}{\hat\Gvar}
\newcommand{\alphat}{\alphavar_t}

\newcommand{\Fhatfnc}{{\cred \mathcal{G}}}
\newcommand{\fhatfnc}{G}
\newcommand{\pivar}{\mathrm{\Gamma}}
\newcommand{\pit}{\pivar_t}

\newcommand{\uplusvar}{u^+}
\newcommand{\uminusvar}{u^-}
\newcommand{\zvar}{z}
\newcommand{\Tvar}{\theta}

\newcommand{\zopt}{z^{MSLP}}

\newcommand{\zuldr}{z^{LDR}}
\newcommand{\zuldrts}{z^{2S}}
\newcommand{\zusaa}{\hat{z}^{2S}_{N}}
\newcommand{\zlldr}{v^{LDR}}
\newcommand{\zlldrts}{v^{2S}}
\newcommand{\zlsaa}{\hat{v}^{2S}_{N}}
\newcommand{\dimp}{\tau_P}
\newcommand{\dimd}{\tau_D}

\newcommand{\Bset}{\mathcal{B}}
\newcommand{\Aset}{{\cred \mathcal{L}}}
\newcommand{\BTset}{\mathcal{D}}

\newcommand{\MSLP}{MSLP}

\newcommand{\DMSLP}{D-MSLP}
\newcommand{\LDRprimal}{{\cred P-LDR}}
\newcommand{\LDRdual}{{\cred D-LDR}}
\newcommand{\PartialLDRprimal}{{\cred P-LDR-2S}}
\newcommand{\PartialLDRdual}{{\cred D-LDR-2S}}
\newcommand{\MSIP}{MSIP}

\newcommand{\new}[1]{{\cred #1}}

\usepackage{hyperref}	 
\hypersetup{
 colorlinks={true},linkcolor={magenta},citecolor=blue!60!black,urlcolor  = blue!60!black,
 }

\title{Two-stage Linear Decision Rules for Multi-stage Stochastic Programming}
\author[1]{Merve Bodur\thanks{bodur@mie.utoronto.ca}}
\author[2]{James R. Luedtke\thanks{jim.luedtke@wisc.edu}}
\affil[1]{\small Department of Mechanical and Industrial Engineering, University of Toronto} 
\affil[2]{\small Department of Industrial and Systems Engineering, University of Wisconsin-Madison}

\begin{document}

\onehalfspacing

\maketitle

\begin{abstract}
\noindent
Multi-stage stochastic linear programs (MSLPs) are notoriously hard to solve in general. Linear decision rules (LDRs)
yield an approximation of an MSLP by restricting the decisions at each stage to be an affine function
of the observed uncertain parameters. Finding an optimal LDR is a static optimization
problem that provides an upper bound on the optimal value of the MSLP, and, under certain assumptions, can be formulated as an explicit linear
program.
Similarly, as proposed by Kuhn, Wiesemann, and Georghiou (``Primal and dual linear decision rules in
stochastic and robust optimization'' {\it Math. Program.} 130, 177--209, 2011) a lower bound for an MSLP can be obtained by
restricting decisions in the dual of the MSLP to follow an LDR. 
We propose a new approximation approach for MSLPs, \emph{two-stage LDRs}. The idea is to require only the 
state variables in an MSLP to follow an LDR, which is sufficient to obtain an approximation of an MSLP that is a {\it two-stage stochastic
linear program} (2SLP). We similarly propose to apply LDR only to a subset of the variables in the dual of the MSLP, which 
yields a 2SLP approximation of the dual that provides a lower bound on the optimal value of the MSLP. Although solving the corresponding 2SLP approximations exactly is intractable
in general, we investigate how approximate solution approaches
that have been developed for solving 2SLP can be applied to solve these approximation problems, and derive statistical
upper and lower bounds on the optimal value of the MSLP.  In addition to
potentially yielding better policies and bounds, this approach requires many fewer assumptions than are required to obtain an explicit
reformulation when using the standard static LDR approach. A computational study on two example problems demonstrates that using a two-stage LDR can yield significantly better primal policies and modestly better dual policies than using policies based on a static LDR.
\end{abstract}


\section{Introduction}
\label{sec:intro}

We present a new approach for approximately solving multi-stage stochastic linear programs (\MSLP s). \MSLP s
model dynamic decision-making processes in which a decision is made, a stochastic outcome is observed, another
decision is made, and so on, for $T$ stages. At each stage, the decision vectors are constrained by linear constraints that depend on the history
of observed stochastic outcomes. A solution of \new{an} \MSLP \ is a policy, which defines the decisions to be
made at each stage as a function of the observed outcomes up to that stage. The objective in an \MSLP \ is to choose a
policy that minimizes the expected cost over all stages. 
Although MSLPs can be used to model a wide variety of problems (e.g., \cite{wallace-ziemba:05}), they are notoriously hard to solve in general \cite{dyer2006computational,shapiro2005complexity}. 

There are a variety of methods available for MSLPs in the case that the stochastic process is represented by a 
scenario tree \cite{caseysen:mor05,Heitsch2009,hoyland2001generating,shapiro2014lectures}.  
Such algorithms include nested Benders decomposition \cite{birge1985decomposition,Birge1996,gassmann:90}, progressive hedging
\cite{rockafellar.wets:91}, and aggregation and partitioning
\cite{boland2016scenario,birge1985aggregation}, and enable the solution of \MSLP s with possibly very large scenario
trees. Unfortunately,  as discussed in \cite{shapiro2005complexity}, the size of a scenario tree needed to obtain even a
modestly accurate approximation  grows exponentially in the number of stages. For example, a $10$-stage problem in which
the uncertainty in each stage is represented by just $50$ realizations would yield a scenario tree having nearly $2 \cdot 10^{15}$
scenarios, making any approach that requires even a single pass through the scenario tree impossible. 

Under some conditions, including stage-wise independence of the random variables, stochastic dual dynamic programming
(SDDP) \cite{pereira1991multi} can overcome the difficulty in exploding scenario tree size, by constructing a
single value function approximation for each stage.  
The SDDP algorithm
converges almost surely on a finite scenario tree \cite{chen1999convergent,shapiro2011analysis} (see also
\cite{girardeau,guigues2,philpott2008convergence} for related results). 
In some cases, such as
additive dependence \cite{infanger1996cut} and Markov dependence \cite{philpott2012dynamic}, the assumption of
stage-wise independence can be satisfied via an appropriate reformulation, e.g., see Example 10 in
\cite{shapiro2011topics} (see also \cite{guigues1} for other types of dependence). However, such
reformulations are not applicable for stage-wise dependence of random recourse matrices (i.e., in the coefficients of the
constraints) \new{ or objective coefficients}.  
Similar approaches that exploit stage-wise independence and value function approximations include multi-stage stochastic
decomposition \cite{sen2014multistage} and approximate dynamic programming \cite{powell2007approximate}.


An alternative approach to handling the complexity of \MSLP\ is to restrict the functional form of the policy. 
One such approach is the use of {\it linear decision rules} (LDRs). 
The idea of an LDR is to require that all decisions made in each stage be a linear (or affine) function of the observed
random outcomes up to that stage. The problem then reduces to a static problem of finding the best LDR, whose expected
cost then yields an upper bound on the optimal value of the \MSLP. In this paper, we refer this use of \new{an} LDR as a {\it
static LDR}.
While LDRs have a long history (see, e.g., \cite{garstka1974decision}), they have recently gained renewed interest in the mathematical
optimization literature after their application to adjustable robust optimization in \cite{ben2004adjustable}. 
The adaptation of this approach to \MSLP\ was presented in \cite{shapiro2005complexity}, and Kuhn et
al.~\cite{kuhn2011primal} analyzed the application of a static LDR in the dual of the \MSLP, which yields a lower bound on the
optimal value of the \MSLP. 
Moreover, under certain assumptions, the static approximations obtained after restricting the primal and dual policies
to be \new{an} LDR  are both tractable linear programs, as shown in \cite{shapiro2005complexity} and
\cite{kuhn2011primal}, respectively. The assumptions include stage-wise independence (or a slight generalization),
compact and polyhedral support, and \new{that uncertainty is limited to the right-hand side of the
constraints}. 
While in some cases static LDR policies provide high quality approximations to \MSLP, they have potential to be significantly
suboptimal. Better policies (primal and dual) can be obtained by
considering more flexible (nonlinear) rules such as (static) piecewise linear decision rules \cite{chen2008linear} and polynomial
decision rules \cite{bampou2011scenario}. 

We propose a new use of \new{an} LDR, which we refer to as {\it two-stage linear decision rules}. The
key idea is to partition the decision variables into {\it state} and {\it recourse} decision variables, with the property
that if the state variables are fixed, then the problem decouples into a separate problem for each stage, involving only
recourse decision variables.  
If one applies an LDR {\it only} to the
state variables, then the problem reduces to a {\it two-stage} stochastic linear program (2SLP), in contrast to a
{\it static} problem which is obtained when using a static LDR.  The advantage of two-stage LDRs is that they free the
recourse variables from the LDR requirement, thus
allowing for a potentially improved policy. Indeed, there exist feasible 2SLPs that are {\it infeasible} if one 
enforces \new{an} LDR on the recourse variables \cite{garstka1974decision}.
This idea of reducing an \MSLP\ to a 2SLP is similar to that proposed by Ahmed \cite{shabbirIMAslides}, except that in
\cite{shabbirIMAslides} the state variables are completely decided in the first-stage and fixed, whereas we allow them to vary
according to an LDR.  We also consider applying a two-stage LDR in the dual of an \MSLP, exploiting the observation that
imposing an LDR restriction {\it only} on the dual variables associated with the \emph{state equations} is
sufficient to obtain a 2SLP that approximates the multi-stage dual problem. We investigate how approximate solutions to the associated
primal and dual approximation problems can be used to obtain feasible policies with associated statistical estimates
on the optimality gap. Our analysis suggests that this can be done under mild assumptions, for example 
that the primal problem exhibits relatively complete recourse (i.e., for any current state there exists a feasible next
decision and state) and has a bounded feasible region with probability $1$. We illustrate 
the two-stage LDR approach \new{on two example problems: an inventory planning problem similar to that studied in \cite{ben2004adjustable,kuhn2011primal},
and a capacity expansion problem proposed in \cite{de2013risk}. We find that, for these problems,
using two-stage LDRs yields significantly better primal policies (upper bounds), and modestly improves on the lower bounds,
when compared to using static LDRs. For the capacity expansion problem, we also compare the two-stage LDR policies and
bounds to those obtained using the SDDP algorithm, when run for a similar amount of computational time. We find that the
SDDP algorithm yields similar lower bounds and better policies for this problem, as expected since the SDDP algorithm
is known to converge to an optimal solution. Thus, the two-stage LDR approximation is expected to be useful primarily
for problems where the SDDP algorithm does not apply.}

A significant challenge to using two-stage LDRs is that the resulting 2SLP is in general intractable to solve exactly. 
Indeed, 2SLP is $\#P$-hard
\cite{dyer2006computational,Hanasusanto2016} due to the difficulty in evaluating the expectation of the recourse
function. 
However, as argued in \cite{shapiro2005complexity}, under mild conditions Monte Carlo sampling-based methods can provide solutions of
modest accuracy to a 2SLP (such a statement cannot be made for MSLP). 
Thus, an important benefit of the two-stage LDR approach is that it enables the application of the long
history of research into solving 2SLPs to the multi-stage setting.
While using a sampling-based method may lead to a suboptimal solution of the 2SLP approximations, our hope is that
this suboptimality may be more than offset
by the improvement gained by eliminating the LDR requirement on the recourse decisions that is imposed when using a static
LDR. In addition, when using a
sampling-based approach, the assumptions that are required for obtaining a tractable reformulation when applying \new{a} static LDR are no longer needed. In particular, the random variables need not have polyhedral (or even
bounded) support, the constraint matrices may be random and dependent across time stages, and the LDR may be based on
nonlinear functions of the random variables. 

The two-stage LDR approach we propose can also be applied to certain multi-stage stochastic {\it integer programs}, in which some
of the decision variables are required to be integer valued. In particular, for the primal problem,
the approach applies directly provided integrality restrictions are imposed only on the recourse variables. 
When the state variables 
have integrality restrictions as well, the form of the decision rule applied to the state variables must be modified, but the two-stage approach still
applies. We refer
to \cite{bertsimas2015design} for one possible such decision rule based on piecewise-linear binary functions. We remark
that combining our approach with that of  \cite{bertsimas2015design} would have potential benefit in terms of
both tractability and policy quality, as removing the piecewise-linear binary decision rule requirement from the
recourse variables both eliminates the need to design such a rule, and gives those decisions more flexibility.

The rest of this paper is organized as follows. Section \ref{sec:primalLDR} defines the \MSLP, reviews the static LDR
approach, and presents the proposed two-stage LDR approach, including discussion
of how to solve the approximate problem and obtain statistical upper bounds on the original \MSLP. 
Section \ref{sec:dualLDR} conducts a similar analysis for the dual of an \MSLP, yielding an approach for finding
statistical lower bounds on an \MSLP. We present illustrative {\cred applications} in {\cred Sections
\ref{sec:invplanning} and \ref{sec:capexp}}, and make concluding remarks in Section \ref{sec:conc}.


\section{Primal two-stage linear decision rules}
\label{sec:primalLDR}

We formulate an \MSLP \ with $T \geq 2$ stages as follows, where throughout the paper, for integers $a \leq b$,
$[a,b] := \{a,a+1,\ldots,b\}$ and $[b] := \{1,\ldots,b\}$:
\bsubeq
\label{eqs:MSLP}
\begin{alignat}{2} 
 \min_{\xvar,\svar} \ \ & \Exp \Big[ \sum_{t \in [\Tpar]} \cpar_t(\xit)^\top \xvar_t(\xit) + \hpar_t(\xit)^\top \svar_t(\xit) \Big] \label{eq:MSLP_obj} \\ 
		\text{s.t.} \ \ 
		& \Apar_t(\xit) \svar_t(\xit) + \Bpar_t(\xit) \svar_{t-1}(\xitmone) + \Cpar_t(\xit) \xvar_t(\xit) =
		\bpar_t(\xit), 
		 \quad && t \in [\Tpar], \ \allxi, \label{eq:MSLP_constState}  \\ 
		& \new{(\xvar_t(\xit),\svar_t(\xit)) \in \Xset_t(\xit),} \quad && t \in [\Tpar], \  \allxi \label{eq:MSLP_constOthers} 
\end{alignat}
\esubeq
\new{where for $t \in [\Tpar]$}
\[ \new{\Xset_t(\xit) := \{   \xvar_t \in \R^{\pt}, \ \svar_t \in \R^{\qt} : \Dpar_t(\xit) \svar_t + \Gpar_t(\xit) \xvar_t \geq
\dpar_t(\xit) \} .} \]
Here, $\{\xisubt\}_{t=1}^T$ is a stochastic
process with probability distribution $\mathbb{P}$ and support $\xisupp$, where $\xipar_1 = 1$ for all $\xi \in \xisupp$
(i.e., data in stage $1$ is deterministic), $\xipar_r$ is a random vector taking values in $\R^{\ksubs}$ for $r \in [2,T]$, and $\xit :=
(\xipar_1,\hdots,\xisubt)$ for $t \in [\Tpar]$. {\cred Letting $\ell_1 = 1$,} we denote $\kt := \sum_{\rpar=1}^t \ksubs$ for $t \in
[T]$.  
The $\svar$ and $\xvar$ variables are referred to as {\it state} and {\it recourse} variables, respectively. Similarly,
\eqref{eq:MSLP_constState} and \eqref{eq:MSLP_constOthers} are referred to as {\it state equations} and {\it recourse 
constraints}, respectively. The
objective is to minimize the expected total cost.  
\new{The functions $\bpar_t:\R^{\kt} \rightarrow \R^{\mt}$, $\dpar_t:\R^{\kt} \rightarrow \R^{\nt}$,
$\Apar_t:\R^{\kt} \rightarrow \R^{\mt \times \qt}$, $\Bpar_t:\R^{\kt} \rightarrow \R^{\mt \times q_{t-1}}$,
$\Cpar_t:\R^{\kt} \rightarrow \R^{\mt \times \pt}$, $\Dpar_t:\R^{\kt} \rightarrow \R^{\nt\times\qt}$, and
$\Gpar_t:\R^{\kt} \rightarrow \R^{\nt \times \pt}$ define the random coefficients as a function of $\xit$. Frequently in
the literature, these are assumed to be affine functions of $\xit$, but we will not need this assumption in this work.
In \eqref{eq:MSLP_constState} for $t=1$, we adopt the convention that $\svar_{0}(\xipar^0) = 0$.}
The constraints are required to be
almost surely satisfied with respect to the distribution of the stochastic process, denoted by ``$\allxi$" 
We note that any \MSLP \ can be brought into the form of \eqref{eqs:MSLP} by introducing additional variables and constraints. 
Throughout the paper, we assume that  \eqref{eqs:MSLP} is feasible and has an optimal solution, and we denote its
optimal objective value as $\zopt$. 

\subsection{Static linear decision rules}

A tractable approximation of \MSLP \ can be obtained by restricting the decision policies to a certain form, i.e., by
restricting the decisions to be a special function of the uncertain parameters. A linear decision rule is 
a policy in which the decisions at each stage $t$ are restricted to be a linear function of the observed random variables
$\xit$ up that stage. We refer to the policies in which all the decisions are required to follow an LDR as \emph{static LDR policies}. Specifically, a static LDR policy has the form:
\bsubeq
\label{eqs:LDRxands}
\begin{align}
& \svar_t(\xit) = {\cred \betat \basis_{t}(\xit)},  \label{eq:LDRs} \\
& \xvar_t(\xit) = {\cred \thetat \basis_{t}(\xit)},   \label{eq:LDRx}  
\end{align} 
\esubeq
where {\cred $\thetat \in \R^{\pt \times \Kt}$ and $\betat \in \R^{\qt \times \Kt}$} are free parameters of the LDR, and
{\cred $\basis_{t}(\xit) = (\basis_{t1}(\xit),\ldots,\basis_{t\Kt}(\xit)):
\R^{\kpar^t} \ra \R^{\Kt}$} for all $t \in [\Tpar]$ \new{is a vector of given} {\it LDR basis functions}. We refer to the
{\cred $\thetavar$} and $\betavar$ variables as the \emph{LDR variables}. We assume $\Kpar_1 = 1$ and $\basis_{t1}(\xipar^t) \equiv 1$ for {\cred all}
$t \in [T]$. 
Often, the basis functions are the
uncertain parameters themselves, i.e., $\Kt = \kt$ and $\basis_{tk}(\xit) = (\xit)_k$, where $(\xit)_k$ denotes the
$k^{\text{th}}$ component of $\xit$ vector. In this case, we refer to the basis functions as the {\it standard basis
functions}. Note that the convention $\xi_1 \equiv 1$ implies that the decisions made at stage $t$ are actually {\it affine} functions of
the random variables $(\xi_2,\ldots,\xi_t)$. \new{Finally, for notational convenience we adopt the convention that $\basis_0 \equiv 0$, so that any
term involving $\basis_0$ disappears.}

Substituting the LDRs of the form \eqref{eqs:LDRxands} into \MSLP \ given in \eqref{eqs:MSLP} yields the following approximation of \MSLP, which we call \LDRprimal:
\begin{alignat}{2} 
\min_{\thetavar,\betavar} \ \ & \Exp \Big[ \sum_{t \in [\Tpar]} \cpar_t(\xit)^\top \thetat \basis_{t}(\xit)
+ \hpar_t(\xit)^\top \betat \basis_t(\xit)  \Big]  \label{eq:ldr} \\ 
		\text{s.t.} \ \ 
		&  \Apar_t(\xit) \betat \basis_t(\xit) + \Cpar_t(\xit) \thetat \basis_t(\xit) +  \Bpar_t(\xit) \betavar_{t-1} \basis_{t-1}(\xitmone) = \bpar_t(\xit),
		\quad  && t \in [\Tpar], \ \allxi,  \nonumber \\ 
		&  \Dpar_t(\xit) \betat \basis_t(\xit) + \Gpar_t(\xit) \thetat \basis_t(\xit)  \geq
		\dpar_t(\xit), \quad && t \in [\Tpar], \ \allxi,  \nonumber \\ 
		& \thetat \in \R^{\pt \times \Kt}, \ \betat \in \R^{\qt \times \Kt}, \quad && t \in [\Tpar]. \nonumber
\end{alignat}
We let $\zuldr$ denote the optimal value of \LDRprimal, where 
here and elsewhere, we adopt the convention that if a minimization (maximization) problem is infeasible, the associated optimal value
is defined to be $+\infty$ ($-\infty$).
Note that all the decision variables are deterministic, i.e., they have to be determined before observing any random
outcomes, and hence this problem is a static problem. \LDRprimal \  is a semi-infinite program having infinitely many
constraints. It is observed in \cite{shapiro2005complexity} (see also \cite{chen2008linear,kuhn2011primal}) that \LDRprimal \ can be
reformulated as a linear program (LP) using robust optimization techniques under the following assumptions: 
\BI
\setlength{\itemindent}{0.2cm}
\I[A1.] The standard basis functions are used.
\I[A2.] For all $t \in [\Tpar]$, the constraint matrices, $\Apar_t,\Bpar_t,\Cpar_t,\Dpar_t,\Gpar_t$,  are independent of
the random vector $\xiT$, \new{and $\bpar_t(\xit)$ and $\dpar_t(\xit)$ are affine functions of $\xit$.}
\I[A3.] The support, $\Xi$, is a nonempty compact polyhedron.
\EI
The LP reformulation has constraints of the form $(\betavar,{\cred \thetavar},w) \in \BTset$, where $w$ are auxiliary variables
and $\BTset$ is an explicitly given polyhedron.
The size of this LP scales well (typically grows only quadratically) with the number of stages $T$. Moreover, the
LP does not require any discretization of $\mathbb{P}$ (e.g., by Monte Carlo sampling), and instead only uses a polyhedral description of $\Xi$
and the second order moment matrix of the random variables.  These results have been generalized in \cite{georghiou2015generalized} to the case of
conic support, where {\cred A3} is replaced by an assumption that $\Xi$ is described by a finite set of conic inequalities, in which case \LDRprimal \ \eqref{eq:ldr} is reformulated as a conic program.

As \LDRprimal \ \eqref{eq:ldr} is a restriction of \MSLP, it provides an upper bound to \MSLP. However, the benefit of tractability
comes at the expense of loss of optimality. That is, the obtained upper bound can be substantially far from the optimal
value of \MSLP.  Indeed, for 2SLPs, the optimal recourse
decisions are very rarely linear in the random variables, but there always exists an optimal  \emph{piecewise linear} decision rule \cite{garstka1974decision}. 

\subsection{Two-stage linear decision rules}

We propose \emph{two-stage LDRs} which yield upper bounds to \MSLP \ that cannot be worse than the ones obtained by
\LDRprimal \ \eqref{eq:ldr}.  The key idea is to apply \new{an} LDR only on the state variables to obtain a {\it two-stage} approximation of
\MSLP, rather than a static approximation. Substituting the LDR of the form \eqref{eq:LDRs} into the \MSLP \ given in \eqref{eqs:MSLP} yields
\begin{alignat}{2} 
\min_{\xvar,\betavar} \ \ & \sum_{t \in [\Tpar]} \Exp \big[ \cpar_t(\xit)^\top \xvar_t(\xit) \big] +  \sum_{t \in
[\Tpar]}   \Exp \big[ \hpar_t(\xit)^\top \betat \basis_t(\xit) \big] \label{eq:ldr2s} \\ 
		\text{s.t.} \ \ 
		& \Apar_t(\xit) \betat \basis_{t}(\xit) + \Cpar_t(\xit) \xvar_t(\xit) + \Bpar_t(\xit) \betavar_{t-1} \basis_{t-1}(\xitmone)   = \bpar_t(\xit), \quad && 
		 t \in [\Tpar], \ \allxi, \nonumber \\ 
		& \Dpar_t(\xit) \betat \basis_{t}(\xit) + \Gpar_t(\xit) \xvar_t(\xit) \geq \dpar_t(\xit),  &&
		 t \in [\Tpar], \ \allxi, \nonumber \\
		& \xvar_t(\xit) \in \R^{\pt},  &&  t \in [\Tpar], \ \allxi, \nonumber \\
		& \betat \in \R^{\qt \times \Kt},  &&  t \in [\Tpar]. \nonumber
\end{alignat}
We denote this problem as \PartialLDRprimal \ and the optimal value of this problem as $\zuldrts$.
This problem is a 2SLP, which
can be equivalently written as follows, where we drop dependence of the first-stage variables on $\xipar^1\equiv 1$ and use
$\betavar_{1} = \svar_1(\xipar^1) = \basis_{1}(\xipar^1) \betavar_{1}$, 
\begin{align*} 
\zuldrts  = \ \min_{\xvar_1,\betavar} \ \ & \cpar_1^\top \xvar_1 +  \sum_{t \in [\Tpar]}  \Exp \big[ 
\hpar_t(\xit)^\top \betat \basis_t(\xit) \big] +  \Exp [ \Ffnc(\betavar,\xiT) ] \\ 
		\text{s.t.} \ \ 
		& \Apar_1 \betavar_{1} + \Cpar_1 \xvar_1 = \bpar_1,  \\ 
		& \Dpar_1 \betavar_{1} + \Gpar_1 \xvar_1 \geq \dpar_1,  \\
		& \xvar_1 \in \R^{\ppar_1}, \\ 
		& \betat \in \R^{\qt \times \Kt}, \ t \in [\Tpar],
\end{align*}
where $\Ffnc(\betavar,\xiT) := \sum_{t \in [2,\Tpar]} \ffnc_t(\betavar,\xit)$ and for $t \in [2,\Tpar]$,
\bsubeq
\label{eq:saasubprob}
\begin{alignat}{2} 
	 \ffnc_t(\betavar,\xit) := \min_{\xvar_t} \ \ & \cpar_t(\xit)^\top \xvar_t \\ 
		 \text{s.t.} \ \ & 
		 \Cpar_t(\xit) \xvar_t = \bpar_t(\xit) - \Apar_t(\xit) \betat \basis_{t}(\xit) - \Bpar_t(\xit) \betavar_{t-1}  \basis_{t-1}(\xitmone), \quad  &&\label{eq:Ffnc_cons1}  \\ 
		& \Gpar_t(\xit) \xvar_t \geq \dpar_t(\xit) - \Dpar_t(\xit) \betat \basis_{t}(\xit), \quad &&\label{eq:Ffnc_cons2}   \\
		& \xvar_t \in \R^{\pt}. \quad &&  \label{eq:Ffnc_cons3} 
\end{alignat}
\esubeq

The following proposition, immediate from the definitions of the associated problems, summarizes the relationship
between the optimal values of \MSLP, \LDRprimal \ \eqref{eq:ldr}, and \PartialLDRprimal \ \eqref{eq:ldr2s}.
\begin{prop}
The following inequalities hold:
\[  \zopt \leq \zuldrts \leq \zuldr . \] 
\end{prop}

The difference between $\zuldr$ and $\zuldrts$ can be arbitrarily large. In particular, an example is given in
\cite{garstka1974decision} of a 2SLP having relatively complete recourse for which  
\LDRprimal  \ \eqref{eq:ldr} is infeasible ($\zuldr = \infty$), while the 2SLP (and hence \PartialLDRprimal,
\eqref{eq:ldr2s}) is feasible. 

Unfortunately, the techniques used to derive a static approximation of \LDRprimal \ \eqref{eq:ldr} do not yield an efficiently computable
reformulation of \PartialLDRprimal \ \eqref{eq:ldr2s}, even under assumptions {\cred A1-A3}. 
%
In the next section, we review approaches for obtaining an approximate solution, say $\hat{\beta}$, of
\PartialLDRprimal \ \eqref{eq:ldr2s}. Then, in Section \ref{sec:ubest} we discuss techniques for \new{obtaining a
feasible policy (and hence estimating an upper bound on $\zopt$)} using such a
solution.

\subsection{Approximate solution of \PartialLDRprimal}
\label{sec:SAAprimal}

There is a huge literature on (approximately) solving 2SLP problems. In this section, we present a brief overview of
relevant approaches, with a focus on identifying the required assumptions. We refer the reader to \cite{shapiro2014lectures} for more details. 

A common approach for approximately solving a 2SLP is  {\it sample average approximation}, in which  $\mathbb{P}$ is
approximated by a discrete probability measure $\hat{\mathbb{P}}$ that assigns positive weights only to a finite
(relatively small) number of realizations of $\xiT$ which are
called \emph{scenarios}. In this way, the intractable expectation term is replaced with a sum. Scenarios may be constructed by a variety of techniques, such as Monte Carlo, 
quasi-Monte Carlo, and Latin hypercube sampling (e.g.,
\cite{tito:siopt08,koivu:05,lind:saacomp,pennanen:09,shapiro2000rate}). For the purpose of this paper, we consider
only the conceptually simplest case in which scenarios are generated via independent Monte Carlo sampling. 

Let $\xiT_j, j=1,\ldots,N$, be an independent and identically distributed (i.i.d.) random sample of the random vector 
$\xiT$, and define the sample average approximation (SAA) problem:
\bsubeq
\label{prim:saa}
\begin{align} 
\zusaa := \min_{\xvar_1,\betavar} \ \ & \cpar_1^\top \xvar_1 +  \sum_{t \in [\Tpar]}
\Exp \big[  \hpar_t(\xit)^\top \beta_t \basis_{t}(\xit)\big]  + \frac{1}{N} \sum_{j \in [N]}
\Ffnc(\betavar,\xiT_j) \label{primsaa:obj} \\ 
		\text{s.t.} \ \ 
		& \Apar_1 \betavar_{1} + \Cpar_1 \xvar_1 = \bpar_1,  \\ 
		& \Dpar_1 \betavar_{1} + \Gpar_1 \xvar_1 \geq \dpar_1,  \\
		& \xvar_1 \in \R^{\ppar_1}, \\ 
		& \betat \in \R^{\qt \times \Kt}, \ t \in [\Tpar].
\end{align}
\esubeq
Once the sample is fixed, the SAA problem can be solved by any approach for solving the above, now deterministic,
problem. In particular the $L$-shape decomposition algorithm \cite{van1969shaped} or a regularized variant
\cite{level-original,Rusz:1986} can be applied, with the further
advantage that the subproblem obtained with fixed $\hbetavar$ decomposes by both scenario and stage due to the relationship
$\Ffnc(\betavar,\xiT) = \sum_{t \in [2,\Tpar]} \ffnc_t(\betavar,\xit)$.
\new{The coefficients on $\betat$ in the objective function \eqref{primsaa:obj}  can also be estimated by sampling in
case the terms $\Exp[\hpar_{tj}(\xit) \basis_{tk}(\xit)]$ cannot be computed efficiently.}

If \new{(i)} there exists a $\bar{\beta}$ such that $\Exp[\Ffnc(\beta,\xiT)] < \infty$ for all $\beta$ in a neighborhood of
$\bar{\beta}$, \new{and (ii) the set of optimal solutions to \PartialLDRprimal  \ \eqref{eq:ldr2s} is nonempty and bounded}, then because $\Ffnc(\cdot,\xiT)$ is a convex function for all $\xiT \in \Xi$, Theorem 5.4 of
\cite{shapiro2014lectures} applies and implies that
\begin{equation}
\label{eq:cons}
\zusaa \rightarrow \zuldrts \quad \text{with probability 1 as }   N \rightarrow \infty 
\end{equation}
and also that the set of optimal solutions to \eqref{prim:saa} converges to the set of optimal solutions of
\PartialLDRprimal  \ \eqref{eq:ldr2s}. 

\new{Stronger results  on the convergence of $\zusaa$ to $\zuldrts$ require additional assumptions. For example, a central limit theorem result (e.g., Theorem 5.7 in
 \cite{shapiro2014lectures}) can
be obtained under the assumptions that $\Exp[\Ffnc(\bar{\beta},\xiT)^2] < \infty$ for some $\bar{\beta}$, and that there
exists a measurable function $f:\Xi \rightarrow \R_+$ such that $\Exp[f(\xiT)^2]$ is finite and
\[ | \Ffnc(\beta,\xiT) -  \Ffnc(\beta',\xiT) | \leq f(\xiT) \| \beta - \beta' \| \]
for all $\beta,\beta'$ and almost every $\xiT \in \Xi$.
Bounds on the sample size required for \eqref{prim:saa} to yield an
$\epsilon$-optimal solution to \PartialLDRprimal \ \eqref{eq:ldr2s} with probability at least $1-\alpha$
are derived in \cite{shapiro:mp08,shapiro2014lectures,shapiro2000rate,shapiro2005complexity}. These bounds scale linearly with the
dimension of the first-stage variables, $\beta$ and $x_1$ in this case.
Dependence on the confidence $\alpha$ is $\ln(1/\alpha)$ so that high confidence can be achieved, but the dependence on $\epsilon$ is $O(1/\epsilon^2)$, which is
why sampling is limited to obtaining ``medium accuracy'' solutions \cite{shapiro2005complexity}. 
These stronger results all require, at least, that  $\Ffnc(\beta,\xiT)$  is finite for every first-stage solution 
$\beta$ and almost every  $\xiT \in \Xi$. 

In order to facilitate the solution of \PartialLDRprimal \ \eqref{eq:ldr2s} via a sampling procedure, we also consider
adding additional constraints $\betavar \in \Bset \subseteq \R^{\dimp}$, where $\dimp := \sum_{t \in [T]} \qpar_t\Kpar_t$, to
\PartialLDRprimal \ \eqref{eq:ldr2s} (and to the SAA \eqref{prim:saa}).  For example, some of the convergence results
require the first-stage feasible region to be bounded, in which case we may define $\Bset$ by limiting the absolute
value of each component of $\betavar$ to be less than a large constant. If the constant is not chosen large enough, then
this may degrade the quality of the solution obtained, but this could be detected after solving the SAA problem by
determining if any of the bound constraints are tight. More significantly, most of the SAA convergence results require
the following relatively complete recourse assumption on the set $\Bset$: 
}
\begin{assumption}
\label{assmp:PartialRecourse}
For all $\betavar \in \Bset$, $\Ffnc(\betavar,\xiT) < +\infty$ $\allxi$.
\end{assumption}
If \PartialLDRprimal \ \eqref{eq:ldr2s} already has relatively complete recourse, then we can take $\Bset = \R^{\dimp}$, and hence impose no additional constraints.
Otherwise, adding the constraints $\beta \in \Bset$ has the potential to make the approximation more conservative. Derivation of a
set $\Bset$ that satisfies Assumption \ref{assmp:PartialRecourse} is a difficult task in general. However, relatively
complete recourse can often be achieved by appropriate modeling, e.g., by introducing ``artificial'' variables that allow violation 
of a constraint, where the violation amount is then penalized in the objective function. Derivation of a set $\Bset$
satisfying Assumption \ref{assmp:PartialRecourse} may then be possible using ad hoc techniques. \new{We provide an
example of how this can be done in an inventory planning problem in Section \ref{sec:invimpl}.}
Another
possibility, if assumptions {\cred A1-A3} hold, is to use the robust optimization techniques used in
\cite{kuhn2011primal,shapiro2005complexity} to derive a tractable set $\BTset$ such that $(\beta,\thetavar)$ is feasible to
\LDRprimal \ \eqref{eq:ldr} if and only if there are values of auxiliary variables $w$ such that $(\beta,\thetavar,w) \in \BTset$. Then, $\Bset = \projx_{\beta}(\BTset)$ would satisfy Assumption
\ref{assmp:PartialRecourse}. This construction of $\Bset$ is more conservative than necessary, because it restricts 
$\beta$ to values for which there is also a $\thetavar$ that makes the static
LDR policy defined in \eqref{eqs:LDRxands} feasible $\allxi$.  However, the resulting
policy could still potentially be better (and for sure would not be worse) than the static LDR policy obtained from
\LDRprimal \ \eqref{eq:ldr},
since enforcing $\beta \in \projx_{\beta}(\BTset)$ would not require the recourse decisions to follow an LDR policy
(it only requires existence of a feasible LDR policy).



\PartialLDRprimal \ \eqref{eq:ldr2s}, with the additional constraints $\betavar \in \Bset$, can also be approximately solved by stochastic approximation \cite{robbinsmonro:51} or one of its
robust extensions, e.g., \cite{njls-sa:siopt09,polyakjuditsky:92}, when Assumption \ref{assmp:PartialRecourse} holds
\new{and $\Bset$ is bounded.}

Finally, we remark that if we cannot derive a set $\Bset$ satisfying Assumption \ref{assmp:PartialRecourse}, results about sampling-based approximation of chance-constrained programs derived in
\cite{campicalifscen} can be used to show that an optimal solution of the SAA problem \eqref{prim:saa} yields a 
policy that is feasible for a large fraction of the random outcomes. This has been previously used in
\cite{bertsimas2007adaptability,bertsimas2015design,vayanos2012} when using
sampling to approximately solve static approximations derived from finitely adaptable and piecewise-linear decision
rules. Although the two-stage LDR policy itself is not necessarily feasible $\allxi$~in this case, in the next section
we discuss how an approximate solution $\hbetavar$ could still be used to guide a feasible policy.

\subsection{Feasible policies and upper bounds on $\zopt$}
\label{sec:ubest}

Let $(\hxvar_1,\hbetavar)$ be an approximate first-stage solution to \PartialLDRprimal \ \eqref{eq:ldr2s}.  \new{We discuss how such a solution
can be used to obtain a feasible policy for the \MSLP \ \eqref{eqs:MSLP}, which in turn can be used to estimate an upper
bound on $\zopt$.
We consider two possibilities for obtaining such a policy, depending on whether or not a set $\Bset$ satisfying
Assumption  \ref{assmp:PartialRecourse} is used.} 

\new{We first consider the case that $\hbetavar \in \Bset$ for a set $\Bset$ satisfying Assumption \ref{assmp:PartialRecourse}. In this
case, $(\hxvar_1,\hbetavar)$ defines a feasible solution to \PartialLDRprimal \ \eqref{eq:ldr2s} and a feasible two-stage LDR policy for
\MSLP. 
In particular, at stage $t$, if the current history is $\xit$, the state variable decisions are given by using
$\hbetavar$ in the LDR \eqref{eq:LDRs} and the recourse decisions are obtained by solving \eqref{eq:saasubprob}, again substituting $\hbetavar$
for $\betavar$.}
\new{As this solution defines a feasible policy}, the expected cost of this solution provides an upper
bound on $\zuldrts$ and $\zopt$.
The expected cost of the policy defined by $(\hxvar_1,\hbetavar)$ can be estimated by generating an independent sample of $\xiT$, say
$\{\xiT_j\}_{j=1}^{N'}$ , and computing  
$$\cpar_1^\top \hxvar_1 +  \sum_{t \in [\Tpar]}   \Exp \big[ \hpar_t(\xit)^\top \hbetavar_t \basis_{t}(\xit)  \big]
 + \frac{1}{N'} \sum_{j \in [N']} \Ffnc(\hbetavar,\xiT_j) .$$ 
Because $\hbetavar$ is fixed in this evaluation step, it would generally be computationally feasible to use
$N' \gg N$.
The values $\Ffnc(\hbetavar,\xiT_j)$ for $j \in
[N']$ can also be used to construct a confidence interval on the objective value of $(\hxvar_1,\hbetavar)$, and hence
a statistical upper bound on $\zopt$.

We next consider the case when we do not know $\hbetavar \in \Bset$ for a set $\Bset$ satisfying Assumption \ref{assmp:PartialRecourse}, 
so that we do not know  a priori that \new{the two-stage LDR defined by $\hbetavar$ defines a feasible policy}. To construct a policy in this case, we make the following relatively complete recourse assumption for the
original problem \MSLP.
\begin{assumption}
\label{assmp:MSLPrecourse}
For all $\xiT \in \Xi$, and each $t \in [2,\Tpar]$, if the random vectors $$\{ (\svar_{r}(\xipar^{r}), \xvar_{r}(\xipar^{r}) \}_{r \in [t-1]}$$ satisfy the
constraints of \MSLP \ for $r \in [t-1]$, then there exists $(\svar_{t}, \xvar_{t})$ that
satisfies the constraints of \MSLP \ in stage $t$:
\begin{align*}
&  \Apar_t(\xit) \svar_t + \Cpar_t(\xit) \xvar_t =
		\bpar_t(\xit) - \Bpar_t(\xit) \svar_{t-1}(\xitmone), \\
& (\xvar_t,\svar_t) \in \Xset_t(\xit) . 
\end{align*}
\end{assumption}
In other words, this assumption states that in any stage $t$, for any value of the previous state variables
$\svar_{t-1}(\xitmone)$ that could be obtained from past realizations of the random outcomes and past feasible
decisions, there always exists a feasible set of decisions in the current stage (see e.g., \cite{higle2006multistage}).  

\new{Under Assumption \ref{assmp:MSLPrecourse}, we can implement} a policy which is guided by $\hbetavar$, which we refer to as a
{\it state-target tracking (STT) policy}.
Specifically, at stage $t=1$, we implement the solution $\laxvar_1 = \hxvar_1$ and $\lasvar_1 = \hbetavar_{1}$.
Then, for each stage $t \in [2,\Tpar]$, we first observe $\xipar_t$ (thus, we have $\xit$), 
and then  solve the problem (deterministic for this fixed $\xit$): 
\bsubeq
\begin{alignat}{1} 
	 \min_{\xvar_t,\svar_t} \ \ & \cpar_t(\xit)^\top \xvar_t + \hpar_t(\xit)^\top \svar_t  + \rho \| \svar_t -
	 \hbetavar_t \basis_t(\xit) \|  \\
		\text{s.t.} \ \ 
		& \Apar_t(\xit) \svar_t + \Cpar_t(\xit) \xvar_t = \bpar_t(\xit) - \Bpar_t(\xit) \lasvar_{t-1}(\xitmone),
		\nonumber \\ 
		& (\xvar_t,\svar_t) \in \Xset_t(\xit),
\end{alignat}
\esubeq
where  $\rho \geq 0$ is a parameter of the policy and $\| \cdot \|$ is any norm,
and let the {\cred optimal }solution be $\laxvar_t(\xit),\lasvar_t(\xit)$. For any $\xiT \in \Xi$, all problems in this sequence 
are feasible
when Assumption \ref{assmp:MSLPrecourse} holds, and hence this yields a feasible policy to MSLP. Observe that when $\rho =
0$, the policy reduces to a pure myopic policy that only considers the cost of decisions in each stage, without
considering the impact of $\svar_t$ on future costs. Using larger values of $\rho > 0$ has the effect of encouraging
the decisions to be made in a way that keeps the state close to what would have been achieved if we could exactly follow the LDR policy
defined by $\hbetavar$ on the state variables.  
The cost of the STT policy under a {\cred realization} $\xi$ of the stochastic process is 
$$\sum_{t \in [\Tpar]} \bigl(\cpar_t(\xit)^\top \laxvar_t(\xit) + \hpar_t(\xit)^\top \lasvar_t(\xit) \bigr).$$
The expected cost of the STT policy is an upper bound on the optimal value of \MSLP, and a confidence interval on this
expected cost can be obtained by simulation with independent replications. \new{We do not know an a priori upper bound
on the optimality gap between the expected cost of the STT policy and the optimal value $\zopt$. However,  the dual two-stage LDR discussed in Section \ref{sec:dualLDR} may be used to estimate a lower bound on $\zopt$, which can be used to
provide an a posteriori statistical bound on the optimality gap of the STT policy.} The value of the parameter $\rho$
can be selected by estimating the expected cost of the policy under different values of $\rho$ and choosing
the most promising value, or by using optimization via simulation techniques \cite{fu:ijoc02,hongnelson:wsc09}. \new{For
example, in our numerical experiments, we used a fixed relatively small sample ($N'=100$), and applied a
variant of a golden section algorithm  to find a value of $\rho$ that approximately minimizes the
estimated cost given by this sample. See Section \ref{sec:capimpl} for more details. Once the value of $\rho$ is chosen,
the expected cost of the resulting policy is  evaluated using a larger sample.}
Note that using the STT policy, even the decisions $\lasvar_t(\xit)$ may not necessarily have the
form of an LDR. Thus, simulating this policy yields an estimate of an upper bound on $\zopt$, but not necessarily
on $\zuldrts$.

\exclude{
Assumption \ref{assmp:MSLPrecourse} suggests that a feasible solution can be generated in a sequential fashion using a
one-step look-ahead strategy, as is
common, e.g., in approximate dynamic programming or reinforcement learning
\cite{bertsekas:adp,bertststits:ndp,powell:adp,suttonbarto:rl}.
Such an approach starts from the first stage and moves forward
the optimization horizon in each stage, assuming that the status of the system is updated as the uncertain
parameters have been revealed. As the stages proceed, the decisions in the current stage are computed based on the current
observed random outcome and value of the state variables.

}

\section{Dual two-stage linear decision rules}
\label{sec:dualLDR}
In this section, we apply \new{a} two-stage LDR to the dual of \MSLP, with the goal of obtaining lower bounds on the
optimal value of \MSLP. The dual of \MSLP, which we refer to as \DMSLP,  is the problem (see
\cite{eisner1975duality}):
\bsubeq
\label{eqs:DMSLP}
\begin{alignat} {2}
	 \max_{\Lvar,\Gvar} \ \ & \Exp \Big[ \sum_{t \in [\Tpar]} \bpar_t(\xit)^\top \Lvart + \dpar_t(\xit)^\top \Gvart \Big] \label{eq:DMSLP_obj} \\ 
		\text{s.t.} \ \ 
		& \Exp \Big[ \Bpar_{t+1}(\xitpone)^\top \Lvar_{t+1}(\xitpone) \Big | \xit \Big] \nonumber \\
		& \qquad +  \Apar_t(\xit)^\top \Lvart + \Dpar_t(\xit)^\top \Gvart =  \hpar_t(\xit), \quad &&  t \in [\Tpar], \ \allxi, \label{eq:DMSLP_con1} \\
		& \Cpar_t(\xit)^\top \Lvart + \Gpar_t(\xit)^\top \Gvart  =  \cpar_t(\xit), \quad &&  t \in [\Tpar], \ \allxi, \label{eq:DMSLP_con2} \\
		& \Gvart \geq 0, \quad &&  t \in [\Tpar], \ \allxi,  \label{eq:DMSLP_con4} \\
		& \Lvart \in \R^{\mt}, \ \Gvart \in \R^{\nt}, \quad &&  t \in [\Tpar], \ \allxi,  \label{eq:DMSLP_con5} 
\end{alignat}
\esubeq
where $\Bpar_{\Tpar+1}(\xipar^{\Tpar+1}) = 0$. For $t \in [T]$, the dual decisions $\Lvar_t(\cdot)$ (corresponding to constraints \eqref{eq:MSLP_constState}
in \MSLP) and $\Gvar_t(\cdot)$ (corresponding to constraints \eqref{eq:MSLP_constOthers} in \MSLP) are functions of the data $\xit$ observed up to
stage $t$. 
Weak duality holds for \MSLP \ and \DMSLP, i.e., the optimal objective value of \DMSLP \ provides a lower bound on $\zopt$. Moreover, under some conditions, strong duality holds (i.e., optimal value of \DMSLP \ equals $\zopt$) 
\cite{eisner1975duality}, although we only require weak duality.

\subsection{Static linear decision rules}

In \cite{kuhn2011primal} it has been proposed to use \new{a} static LDR to obtain a tractable approximation of \DMSLP, and thus
an efficiently computable lower bound on $\zopt$. Specifically, the idea is to require all the dual decisions to be
\new{an} LDR, i.e.,  
\bsubeq
\label{eqs:LDR_LandG}
\begin{align}
& \Lvar_t(\xit) = {\cred \alphat \basis_{t}(\xit)},  \label{eq:LDR_L} \\
& \Gvar_t(\xit) = {\cred \pit \basis_{t}(\xit)},  \label{eq:LDR_G} 
\end{align} 
\esubeq
where  $\alphat \in \R^{\mt \times \Kt}, \pit \in \R^{\nt \times \Kt}$, for all $t \in [\Tpar]$ are the parameters of the
decision rule. Imposing \eqref{eqs:LDR_LandG} yields the following static approximation of  \DMSLP, which we call \LDRdual:
\begin{alignat} {2}
	 \max_{\alphavar,\pivar} \ \ & \Exp \Big[ \sum_{t \in [\Tpar]}  \bpar_t(\xit)^\top \alphat \basis_{t}(\xit) 
	 + \dpar_t(\xit)^\top \pit \basis_t(\xit)  \Big]  \label{eq:dldr} \\ 
		\text{s.t.} \ \ 
		& \Exp \Big[ \Bpar_{t+1}(\xitpone)^\top \alphavar_{t+1} \basis_{t+1}(\xitpone) \Big | \xit
		\Big] \nonumber \\ 
		& \ \ \ +  \Apar_t(\xit)^\top \alphat \basis_{t}(\xit) + \Dpar_t(\xit)^\top \pit \basis_{t}(\xit) =
		\hpar_t(\xit), \ \ \ &&  t \in [\Tpar], \ \allxi, \nonumber \\
		&    \Cpar_t(\xit)^\top \alphat \basis_{t}(\xit)+ \Gpar_t(\xit)^\top \pit \basis_{t}(\xit)  =
		\cpar_t(\xit), \quad &&  t \in [\Tpar], \ \allxi, \nonumber \\
		&  \pit \basis_{t}(\xit) \geq 0, \quad &&  t \in [\Tpar], \ \allxi, \nonumber \\
		& \alphat \in \R^{\mt \times \Kt}, \ \pit \in \R^{\nt \times \Kt}, \quad && t \in [\Tpar]. \nonumber
\end{alignat}
We refer to the optimal value of \LDRdual\ as $\zlldr$.
The semi-infinite program \LDRdual \ \eqref{eq:dldr} can be reformulated as an efficiently
solvable LP \new{if assumptions A1-A3 stated in Section \ref{sec:primalLDR} hold, the problem \MSLP \ is
strictly feasible, and the following additional assumption holds
 \cite{kuhn2011primal}:}
\BI
\I[A4.] {\cred The conditional expectation $\Exp(\xiT | \xit)$ is  almost surely linear in $\xit$ for all $t \in [\Tpar]$
(e.g., this occurs when $\{ \xisubt\}_{t \in [\Tpar]}$ are mutually independent, known as \emph{stage-wise independence}).}
\EI
If assumption {\cred A3} is replaced by an assumption that $\Xi$ is described by conic inequalities, \LDRdual \ \eqref{eq:dldr} can be reformulated as a conic program \cite{georghiou2015generalized}.

\subsection{Two-stage linear decision rules}

Examining the structure of \DMSLP, given in \eqref{eqs:DMSLP}, we observe that if the values $\Lvar_t(\xit)$ are fixed, then
\eqref{eqs:DMSLP} decomposes by stage.  We thus propose to apply \new{an} LDR only to the $\Lvar$ variables, leaving the
decision variables $\Gvar$ as recourse variables. Imposing the LDR of \eqref{eq:LDR_L}  
collapses \DMSLP \ into the following 2SLP, which
we refer to as \PartialLDRdual: 
\begin{align} 
\zlldrts := \max_{\Gvar_1,\alphavar} \ \ & \dpar_1^\top \Gvar_1 + \sum_{t \in [\Tpar]}  \Exp \big[  \bpar_t(\xit)^\top
\alphat \basis_{t}(\xit) \big]  + \Exp [ \Fhatfnc(\alphavar,\xiT) ]
	 \label{eq:dldr2s} \\ 
		\text{s.t.} \ \ 
		& \Exp \big[ \Bpar_2(\xipar^2)^\top \alphavar_2 \basis_{2}(\xipar^2) \big]  +
		\Apar_1^\top \alphavar_{1} + \Dpar_1^\top \Gvar_1 =  \hpar_1, \nonumber \\
		& \Cpar_1^\top \alphavar_{1} + \Gpar_1^\top \Gvar_1  =  \cpar_1, \nonumber \\
		& \Gvar_1 \in \R^{\npar_1}_+, \nonumber \\
		& \alphat \in \R^{\mt \times \Kt}, \quad  t \in [\Tpar],  \nonumber
\end{align}
where we have dropped the dependence on $\xipar_1 \equiv 1$ on the first-stage decision variables.
Here, $\Fhatfnc(\alphavar,\xiT)$ is the second-stage value function,
\beqn
\label{eq:Fhatseparated}
\Fhatfnc(\alphavar,\xiT) := \sum_{t \in [2,\Tpar]} \fhatfnc_t(\alphavar,\xit)
\eeqn
where for each $t \in [2,\Tpar]$ 
\bsubeq
\label{eqs:fhatfnc}
\begin{align} 
	 \fhatfnc_t(\alphavar,\xit) :=  \max_{\Gvar_t} \   & \dpar_t(\xit)^\top \Gvar_t \label{eq:fhatfnc_obj} \\ 
		\text{s.t.} \ 
		& \Dpar_t(\xit)^\top \Gvar_t =  \hpar_t(\xit) - \Apar_t(\xit)^\top \alphat \basis_{t}(\xit) \nonumber \\
		&  \hspace*{2cm} -  \Exp \Big[ \Bpar_{t+1}(\xitpone)^\top \alphavar_{t+1} \basis_{t+1}(\xitpone) \ \Big | \xit \Big], \label{eq:fhatfnc_cons1} \\
		& \Gpar_t(\xit)^\top \Gvar_t  =  \cpar_t(\xit) - \Cpar_t(\xit)^\top \alphat \basis_{t}(\xit), \label{eq:fhatfnc_cons2} \\
		& \Gvar_t \in \R_+^{\nt}. \label{eq:fhatfnc_cons5}		
\end{align}
\esubeq

The following proposition, immediate from the definitions of the associated problems, summarizes the relationship
between the optimal values of \MSLP, \LDRdual \ \eqref{eq:dldr}, and \PartialLDRdual \eqref{eq:dldr2s} .
\begin{prop}
The following inequalities hold:
\[  z^{MSLP} \geq \zlldrts \geq  \zlldr . \] 
\end{prop}

As \PartialLDRdual \ \eqref{eq:dldr2s} is a 2SLP, the discussion in Section \ref{sec:SAAprimal} of methods to
obtain an approximate solution to \PartialLDRprimal \ \eqref{eq:ldr2s} \ 
applies also to \PartialLDRdual \ \eqref{eq:dldr2s}. In particular, one possibility is to obtain an i.i.d. sample
$\{\xiT_j\}_{j=1}^N$ of $\xiT$ and solve the SAA problem:
\bsubeq
\label{dual:saa}
\begin{align} 
\zlsaa := \max_{\Gvar_1,\alphavar} \ \ & \dpar_1^\top \Gvar_1 + \sum_{t \in [\Tpar]} \Exp \big[  \bpar_t(\xit)^\top \alphat \basis_{t}(\xit)\big]  + \frac{1}{N} \sum_{j \in [N]}
	 \Fhatfnc(\alphavar,\xiT_j)  \\ 
		\text{s.t.} \ \ 
		& \Exp \big[  \Bpar_2(\xipar^2)^\top \alphavar_{2} \basis_{t}(\xipar^2) \big] +
		\Apar_1^\top \alphavar_{1} + \Dpar_1^\top \Gvar_1 =  \hpar_1, \label{dual:saaeq} \\
		& \Cpar_1^\top \alphavar_{1} + \Gpar_1^\top \Gvar_1  =  \cpar_1, \\
		& \Gvar_1 \in \R^{\npar_1}_+, \\
		& \alphat \in \R^{\mt \times \Kt}, \quad  t \in [\Tpar].
\end{align}
\esubeq
As in the primal, note that the SAA problem can be solved by decomposition algorithms as the second-stage problem decomposes by both
scenario and by stage due to the relationship \eqref{eq:Fhatseparated}. \new{The expected value coefficients in the objective and
constraints \eqref{dual:saaeq} may be further estimated by sampling in case they cannot be computed efficiently.}

\subsection{Obtaining lower bounds on $\zopt$}
\label{sec:duallb}

Next, we discuss how to use  
an approximate solution of \PartialLDRdual \  \eqref{eq:dldr2s}  
to estimate a lower bound on $\zopt$. As in the primal case, in order to assure that we obtain a two-stage LDR policy
that is feasible for all possible realizations,  we consider the possibility of adding a set of constraints 
$\alphavar \in \Aset \subseteq \R^{\dimd}$ to \PartialLDRdual \ \eqref{eq:dldr2s} and its SAA counterpart \eqref{dual:saa} where $\dimd :=
\sum_{t \in [\Tpar]} \Kt \mt$. The following assumption on $\Aset$ assures that the problem \PartialLDRdual\ \eqref{eq:dldr2s} has relatively
complete recourse when the constraints $\alphavar \in \Aset$ are enforced.

\begin{assumption}
\label{assmp:DualPartialRecourse}
For all $\alphavar \in \Aset$, $\Fhatfnc(\alphavar,\xiT) > -\infty$ $\allxi$.
\end{assumption}

The following assumption provides a sufficient condition under which the set $\Aset = \R^{\dimd}$ satisfies Assumption
\ref{assmp:DualPartialRecourse} (i.e., no additional constraints are necessary).

\begin{assumption}
\label{assmp:Boundedness}
The set {\cred $X_t(\xit)$} is bounded for all $t \in [\Tpar]$ and $\mathbb{P}$ almost all $\xiT \in \Xi$.
\end{assumption}
A special case of this assumption occurs when $\xvar$ and $\svar$ variables have explicit upper and lower bounds. An
important feature of this assumption is that the sets $X_t(\xit)$ are not required to be uniformly bounded. 
For example, bounds on the decision variables of the form $0 \leq \xvar_t(\xit) \leq M(\xit)$ (and similarly for $\svar$
variables), are sufficient for satisfying this assumption, even if $M(\xit)$ is not bounded over $\xiT \in \Xi$. 

\begin{prop}
If Assumption \ref{assmp:Boundedness} is satisfied, then $\Aset = \R^{\dimd}$ satisfies
Assumption \ref{assmp:DualPartialRecourse}.
\end{prop}

\begin{proof} 
We show that for any given $\alphavar \in \R^{\dimd}$ and $\xiT \in \R^{\kpar^{\Tpar}}$, \eqref{eqs:fhatfnc} is feasible for any $t \in [\Tpar]$. 
Let $R_b(\xit)$ and $R_c(\xit)$ denote the right-hand sides of the constraints \eqref{eq:fhatfnc_cons1} and \eqref{eq:fhatfnc_cons2}, respectively. Then, the dual of \eqref{eqs:fhatfnc} 
\begin{align*} 
	 \min \ \ & R_b(\xit)^\top \svar_t(\xit) + R_c(\xit)^\top \xvar_t(\xit) \\ 
		\text{s.t.} \ \ 
		& {\cred (\xvar_t(\xit),\svar_t(\xit)) \in \Xset_t(\xit)},   \\
		& \xvar_t(\xit) \in \R^{\pt}, \ \svar_t(\xit) \in \R^{\qt},
\end{align*}
is bounded due to Assumption \ref{assmp:Boundedness}. It is also feasible as \MSLP \ is assumed to be feasible. This implies that \eqref{eqs:fhatfnc} cannot be infeasible.  
\end{proof}

Now, suppose we have an approximate solution $(\hGvar_1,\halphavar)$ to \PartialLDRdual \ \eqref{eq:dldr2s},  where $\halphavar \in \Aset$
for some set $\Aset$ that satisfies Assumption \ref{assmp:DualPartialRecourse}. 
In this
case, $(\hGvar_1,\halphavar)$ defines a feasible solution to \PartialLDRdual \ \eqref{eq:dldr2s},
and hence its objective value provides a lower 
bound on $\zlldrts$, and hence is a lower bound on $\zopt$.
The objective value of $(\hGvar_1,\halphavar)$ can be estimated by generating an independent sample of $\xiT$, say
$\{\xiT_j\}_{j=1}^{N'}$ (where possibly $N' \gg N$), and computing  
$$ \dpar_1^\top \hGvar_1 + \sum_{t \in [\Tpar]} \Exp \big[ \bpar_t(\xit)^\top {\halphavar}_{t} \basis_{t}(\xit) \big]  + \frac{1}{N'} \sum_{j \in [N']}
	 \Fhatfnc(\halphavar,\xiT_j)  .  $$
The
values $\Fhatfnc(\halphavar,\xiT_j)$ for $j \in [N']$ can also be used to construct a confidence interval on the
objective value of $(\hGvar_1,\halphavar)$, and hence a statistical lower bound on $\zopt$. 

We close this section by discussing an approach for estimating the gap between a primal two-stage LDR policy defined by
$(\hxvar_1,\hbetavar)$
and a dual two-stage LDR policy defined by $(\hGvar_1,\halphavar)$. Following \cite{makmortwood}, the motivation is that if the same sample (common random numbers) is used
in estimating the upper and lower bounds, then the variance of the gap estimator can be reduced if the
upper and lower bound sample estimates are positively correlated.  Specifically, given a sample $\{ \xiT_j \}_{j=1}^{N'}$, the gap
observations are then calculated as
\begin{align*}
 \mathrm{Gap}_j = & 
\Bigl[ \cpar_1^\top \hxvar_1 +  \sum_{t \in [\Tpar]}  \Exp \big[ \hpar_t(\xit)^\top \hbetavar_{t} \basis_{t}(\xit)  \big]
 + \Ffnc(\hbetavar,\xiT_j) \Bigr] \\ 
& - \Bigl[\dpar_1^\top \hGvar_1 + \sum_{t \in [\Tpar]} \Exp \big[\bpar_t(\xit)^\top \halphavar_{t} \basis_{t}(\xit) \big]  + \Fhatfnc(\halphavar,\xiT_j) \Bigr],
\end{align*}
for $j \in [N']$. These values can then be used to construct a confidence interval on the gap.

{\cred 
\section{Illustrative example: Inventory planning}
\label{sec:invplanning}

We first present a numerical example on an inventory planning problem to investigate the performance of two-stage LDR policies and bounds, in comparison to static LDR
policies and bounds.  

\subsection{Problem description}
\label{subsec:invdesc}

We consider a variation of the inventory planning problem used for numerical illustration in
\cite{ben2004adjustable,kuhn2011primal}. The system consists of $I$ factories and a single product type, and the goal is to meet demands over the
planning horizon at minimum expected cost. The model is stated as follows:
\begin{alignat}{2} 
\min \ & \Exp \bigg[\sum_{t \in [T]} \sum_{i \in [I]} c_{it} x_{it}(\xit) \bigg] \\
\text{s.t. } & s_{t-1}(\xit) - s_{t}(\xit) + \sum_{i \in [I]} x_{it}(\xit)  = \xi_t, &\quad& t \in [\Tpar], \ \allxi,
\label{inv:bal} \\
& \underline{s} \leq s_{it}(\xit) \leq \bar{s}   && t \in [\Tpar], \ i \in [I], \ \allxi,  \label{inv:statelim} \\
& 0 \leq x_{it}(\xit) \leq \bar{x}_{i}   && t \in [\Tpar], \ i \in [I], \ \allxi.  \label{inv:prodlim}
\end{alignat}
Here, $\xi_t$ is a scalar random variable representing demand for the product in each $t \in [T]$. The recourse
decision variable $x_{it}(\xit)$ determines amount of the product to produce in factory $i$ at stage $t$, while the
state variable $s_t(\xit)$ represents the inventory level at the end of stage $t$. Constraints \eqref{inv:bal} are the
inventory balance equations, \eqref{inv:statelim}
limit the inventory level to be between lower bound $\underline{s}$ and upper bound $\bar{s}$, and  \eqref{inv:prodlim} are the limits on production in each
stage to be at most $\bar{x}_{i}$. 

The model in \cite{ben2004adjustable,kuhn2011primal} also has a constraint on the total
amount that can be produced from any single factory over all the stages in the planning horizon. Modeling this constraint in our standard
model format requires introducing an additional state variable for each factory $i$, representing the cumulative amount of production from
each factory. Imposing an LDR on that state variable would in turn imply that the variables
$x_{it}(\xit)$ also follow an LDR, and hence for that model the static and two-stage LDR 
policies are identical. This illustrates an example where there is no benefit to using a two-stage 
LDR over a static LDR. In the version we consider, the $x_{it}(\xit)$ are still flexible when
the state variables $s_t(\xit)$ follow an LDR, and hence there is potential for a two-stage LDR to yield better
solutions.

Following the data in \cite{ben2004adjustable,kuhn2011primal}, we consider an instance with $I=3$,
$\underline{s}=500$, $\bar{s}=2000$, and $\bar{x}_{i}=567$ for $i \in [I]$. The random demand $\xi_t$ in stage $t \in [T]$ is uniformly
distributed in the interval $\Xi_t = [(1-\theta) \xi^*\zeta_t,(1+\theta)\xi^*\zeta_t]$, where $\theta = 0.3$ is the
variability parameter, $\xi^*=1000$ is the nominal demand, and $\zeta_t = 1 + (1/2) \sin (\pi(t-1)/12)$ is the seasonality
factor. Finally, the cost coefficients are defined as $c_{it} = \alpha_i \zeta_t$, where $\alpha_1 = 1$, $\alpha_2=1.5$, and $\alpha_3=2$.

\subsection{Implementation details}
\label{sec:invimpl}
For both the static and two-stage LDR policies, we use the standard basis functions, $\xit$, in stage $t$.
For the static LDR, we implemented the deterministic reformulations proposed in
\cite{kuhn2011primal,shapiro2005complexity} to obtain upper bounds (with a primal LDR policy) and lower bounds (with a
dual LDR policy). 

For the primal two-stage LDR policy, we first observe that this problem as stated does not satisfy relatively complete
recourse, Assumption \ref{assmp:MSLPrecourse}, although we remark that this assumption is satisfied in a slightly
modified version of the problem in which 
variables are introduced to allow some amount of demand to go unserved, with a large penalty. Rather than making this modification,
however, we demonstrate how for this problem a set of constraints satisfying Assumption \ref{assmp:PartialRecourse} can be derived.
Using the standard basis functions, the state variables $s_t(\xit)$ take the form
\[ s_t(\xit) = \betat \xit \]
where $\betat \in \R^{1 \times t}$. Thus, the constraints \eqref{inv:statelim} take the form
\[ \underline{s} \leq \betat \xit \leq \bar{s}, \quad t \in [\Tpar], \forall \xi_t \in  [1-\theta
\xi^*\zeta_t,(1+\theta)\xi^*\zeta_t].  \]
These constraints can be reformulated with deterministic linear constraints in an extended variable space using standard robust optimization techniques
\cite{Ben-Tal2002}.  To ensure $\betat, t \in [T]$ are selected such that constraints \eqref{inv:bal} can be satisfied
for some $x_t(\xit), i \in [I]$ satisfying \eqref{inv:prodlim}, it is sufficient to enforce
\[ \xi_t - \sum_{i \in [I]} \bar{x}_{i} \leq  \beta_{t-1} \xi^{t-1} - \betat \xit  \leq  \xi_t, \quad t \in [\Tpar], \forall \xi_t \in  [(1-\theta)
\xi^*\zeta_t,(1+\theta)\xi^*\zeta_t],  \]
where the lower bound is based on the maximum total production and the upper bound is based on the minimum
total production in each period. Again, these constraints can be reformulated as deterministic linear constraints using
robust optimization techniques. 

For both the primal and dual two-stage LDR policy, we use 250 scenarios to construct an SAA, and solve the resulting problem by explicitly solving the determistic equivalent formulation. Given the
resulting LDR coefficients, we then use an independent sample of $10^5$ scenarios to evaluate the quality of the primal
policy and dual bound.

\subsection{Results}

Table \ref{tab:invres} provides the results comparing the bounds obtained for this instance, for varying values of
$T=2,\ldots,10$. The columns under Static LDR provide the lower bound (LB), upper bound (UB), and optimality gap (Gap
(\%)), respectively, where optimality gap for an instance is calculated as (UB - LB)/UB. For the two-stage LDR policy,
95\% confidence intervals for the lower bound (LB CI) and upper bound (UB CI) are provided, along with an estimate of
the optimality gap, which is computed by using the lower end of the lower bound confidence interval and the upper end of
the upper bound confidence interval. 
\begin{table}[htbp]
  \centering
    \begin{tabular}{lrrrrrrrr}
	 \toprule
	  & \multicolumn{3}{c}{Static LDR} &  \multicolumn{3}{c}{2S LDR} \\ 
	  \cmidrule(lr){2-4}
	  \cmidrule(lr){5-7}
   \multicolumn{1}{c}{$T$} & LB & UB & \multicolumn{1}{c}{Gap (\%)} &
\multicolumn{1}{c}{LB CI}  & \multicolumn{1}{c}{UB CI} &
\multicolumn{1}{c}{Gap (\%)} \\ 
\midrule
    2     & 1972.4 & 2026.0 & 2.65  & 1974.4 $\pm$ 2.7   & 1993.9 $\pm$ 1.9   & 1.21  \\ 
    3     & 3825.0 & 3940.2 & 2.92  & 3831.6 $\pm$ 4.0   & 3856.1 $\pm$ 3.2   & 0.82  \\ 
    4     & 6089.8 & 6345.0 & 4.02  & 6102.4 $\pm$ 5.5   & 6146.9 $\pm$ 4.7   & 0.89   \\ 
    5     & 8664.4 & 9021.3 & 3.96  & 8669.1 $\pm$ 6.6   & 8737.6 $\pm$ 5.9   & 0.93  \\ 
    6     & 11482.4 & 11975.0 & 4.11  & 11515.2 $\pm$ 10.2  & 11594.8 $\pm$ 7.3   & 0.84  \\ 
    7     & 14431.1 & 15076.3 & 4.28  & 14482.3 $\pm$ 12.3  & 14618.8 $\pm$ 8.6   & 1.08  \\ 
    8     & 17431.6 & 18200.3 & 4.22  & 17527.4 $\pm$ 13.7  & 17660.4 $\pm$ 9.9   & 0.89  \\ 
    9     & 20251.8 & 21147.9 & 4.24  & 20326.2 $\pm$ 14.9  & 20535.3 $\pm$ 10.9  & 1.14  \\ 
    10    & 22764.8 & 23738.3 & 4.10  & 22809.5 $\pm$ 15.0  & 23067.0 $\pm$ 11.5  & 1.23  \\ 
	 \bottomrule
    \end{tabular}%
     \caption{Comparison of static and two-stage LDR policies for inventory problem.}
 \label{tab:invres}%
\end{table}%
We find that the two-stage LDR policy can yield modestly better lower bound
estimates than the static LDR lower bounds, and somewhat more significantly better primal policies. In terms of solution
time, the static LDR lower and upper bounds were computed very quickly, less than 0.02 seconds in all cases. For
the two-stage LDR policies, solving the two SAA problems took at most 3.98 seconds, and evaluating the bounds took at
most 5.17 seconds. Thus, as expected, in this case where the assumptions required for obtaining a deterministic
formulation of static LDR apply, the solution time for the static LDR policy are significantly faster than for the
two-stage LDR. On the other hand, the solution times for the two-stage LDR policy were still modest, and yielded better
policies.
} 

\section{Illustrative example: Capacity expansion}
\label{sec:capexp}
\new{We next consider a capacity expansion problem. On this problem, we again compare the two-stage LDR policies and bounds to
those obtained from static LDR policies, and also compare to the policy and lower bound obtained from using the 
SDDP algorithm.}

\subsection{Problem description}
\label{subsec:capexpdesc}
We consider a variant of the stochastic capacity expansion problem given in \cite{de2013risk}. We wish to determine an
investment schedule over $\Tpar$ stages for the installation of new capacities of $\Ipar$ different power generation
technologies, together with some operational decisions to meet demand for power over time. The demand is
modeled by a load duration curve, which is approximated by partitioning each stage into $\Jpar$ segments (of
possibly different length). The demand
corresponding to segment $j \in [\Jpar]$ in $t \in [\Tpar]$ is denoted by $\dpar_{tj}$. The amount of new capacity of
technology $i \in [\Ipar]$ added in stage $t \in [\Tpar]$ is represented by $\uplusvar_{ti}$, and is assumed to be
available for use immediately, i.e., at the beginning of stage $t$. The unit cost of $\uplusvar_{ti}$ is denoted by $c^{\uplusvar}_{ti}$.
The state variable $\svar_{ti}$ represents the current installed capacity of technology $i \in [\Ipar]$ in the beginning of
stage $t \in [\Tpar]$, which incurs holding cost of $c^{\svar}_{ti}$ per unit. We assume that it is possible to discard
(i.e., remove) some capacity of $i \in [\Ipar]$ in $t \in [\Tpar]$, denoted by $\uminusvar_{ti}$, at a (possibly zero) unit cost of
$c^{\uminusvar}_{ti}$. 
The operating level of $i \in [\Ipar]$ at $t \in
[\Tpar]$ for meeting the demand in segment $j \in [\Jpar]$ is represented by the decision variable $\xvar_{tij}$, while
the amount of unsatisfied demand is represented as $\zvar_{tj}$, whose unit costs are $c^\xvar_{tij}$ and $c^\zvar_{tj}$,
respectively. Then, the stochastic capacity expansion problem is formulated as an \MSLP \ as follows:
\bsubeq
\label{eqs:CapExpPrimal}
\begin{alignat}{2} 
	 \min \ \ & \Exp  \sum_{t \in [\Tpar]} \bigg[\sum_{i \in [\Ipar]} \Big( c_{ti}^{\uplusvar} \uplusvar_{ti}(\xit) +
	 c_{ti}^{\uminusvar} \uminusvar_{ti}(\xit ) + c_{ti}^\svar && \hspace*{-0.6cm} \svar_{ti}(\xit)    + \sum_{j \in [\Jpar]}
	 c_{tij}^{\xvar} \xvar_{tij}(\xit) \Big) + \sum_{j \in [\Jpar]} c_{tj}^\zvar \zvar_{tj}(\xit)  \bigg]   \label{eq:CapExpPrimal_obj} \\ 
		\text{s.t.} \ \ 
		& \svar_{ti}(\xit) - \svar_{t-1,i}(\xitmone) - \uplusvar_{ti}(\xit) \ \, +  \uminusvar_{ti}(\xit) = 0, &&  t \in [\Tpar], \ \allxi, \ i \in [\Ipar],  \label{eq:CapExpPrimal_con1} \\
		& \svar_{ti}(\xit) - \xvar_{tij}(\xit) \geq 0, &&  t \in [\Tpar], \ \allxi , \ i \in [\Ipar], \ j \in [\Jpar],  \label{eq:CapExpPrimal_con2} \\
		& \sum_{i \in [\Ipar]} \xvar_{tij}(\xit) + \zvar_{tj}(\xit) \geq d_{tj}(\xit), &&  t \in [\Tpar], \ \allxi, \ j \in [\Jpar],  \label{eq:CapExpPrimal_con3} \\
		& 0 \leq \zvar_{tj}(\xit) \leq d_{tj}(\xit), &\ & t \in [\Tpar], \ \allxi, \ j \in [\Jpar], \label{eq:CapExpPrimal_con4} \\
		&  0 \leq \uplusvar_{ti}(\xit) \leq M_{ti}^{u^+} , &&  t \in [\Tpar], \ \allxi, \ i \in [\Ipar], \label{eq:CapExpPrimal_con5} \\
		&  0 \leq \uminusvar_{ti}(\xit) \leq M_{ti}^{u^-} , &&  t \in [\Tpar], \ \allxi, \ i \in [\Ipar], \label{eq:CapExpPrimal_con6} \\
		&  0 \leq \svar_{ti}(\xit) \hspace*{0.04cm} \leq M_{ti}^{s} , &&  t \in [\Tpar], \ \allxi, \ i \in [\Ipar],
		\label{eq:CapExpPrimal_con7}  \\
		& \xvar_{tij}(\xit) \geq 0, && t \in [\Tpar], \ \allxi, \ i \in [\Ipar], \ j \in [\Jpar],
		\label{eq:CapExpPrimal_con9} 
\end{alignat}
\esubeq
The objective function \eqref{eq:CapExpPrimal_obj} minimizes the expected total cost. The constraints
\eqref{eq:CapExpPrimal_con1}  are the only state equations, which keep track of the available capacity of each
technology. 
Constraints \eqref{eq:CapExpPrimal_con2} limit the operating levels to
the available capacity level, while \eqref{eq:CapExpPrimal_con3} ensure that either demand is met, or unmet demand is
recorded in the $\zvar_{tj}$ variable values.
Constraints \eqref{eq:CapExpPrimal_con4}-\eqref{eq:CapExpPrimal_con7} represent the bounds on the shortfall,
installation, removal, and inventory level variables, respectively.  Note that
\eqref{eq:CapExpPrimal_con7} constitute upper bounds also on the $\xvar$ variables due to \eqref{eq:CapExpPrimal_con2},
and thus this formulation satisfies Assumption \ref{assmp:Boundedness}. In addition, we assume that $M_{ti}^{u^-}
\geq M_{t-1,i}^{s}$ which ensures that this formulation satisfies relatively complete recourse, Assumption
\ref{assmp:MSLPrecourse}, since at stage $t$, given any feasible value of $s_{t-1,i}(\xitmone)$, it is feasible to set
$\uminusvar_{ti}(\xit) = s_{t-1,i}(\xitmone)$ for $i \in [\Ipar]$, $\zvar_{tj}(\xit) = d_{tj}(\xit)$ for $j \in [\Jpar]$ and all remaining
variables to zero.



To the extent possible, we use data from \cite{de2013risk}, which focuses on a German system, although we extend
their $3$-stage example to $\Tpar =
5,10,15,20$. In \cite{de2013risk}, there are $\Ipar = 3$ technologies (coal-fired power plant, combined cycle gas turbine and open cycle gas
turbine). Each stage is divided into $L = 8$ periods, and $W = 5$ wind regimes are considered for
each period. We model this as $\Jpar = L W = 40$ segments at each stage, corresponding to each period/wind regime pair.
For $t \geq 2$, the demand corresponding to the segment $j=(\ell,w) \in [L] \times [W]$ is modeled as 
$$\dpar_{tj}(\xit)= \max \bigg\{\dpar_{0,\ell} \prod_{r=2}^t \xipar_r^g - \eta_w K_t^{w} \prod_{r=2}^t \xipar_r^{w}, 0 \bigg\}$$
where $\dpar_{0,\ell} $ is the base demand value of period $\ell$,
$\xipar_t^g$ is a random variable reflecting the demand growth of stage $t$, $\eta_w$ is the parameter denoting the
wind efficiency, $K_t^{w}$ is the wind power generation target,
and $\xipar_t^{w}$ is a random variable representing the growth in the wind power generation in stage $t$. 
The values of $\dpar_{0,\ell}$ and $\eta_w$ are from Tables 2 and 3 of \cite{de2013risk}, and are reproduced in Appendix \ref{sec:appA}.
We use $K_2^{w} = 36.64$ and $K_t^{w}=45.75$ for all $t \geq 3$. We assume 
$\xi_t^g$ has lognormal distribution with $\mu=0.2$ and $\sigma = 0.1+0.01t$, 
and $\xi_t^{w}$ has lognormal distribution with $\mu=0.15$ and $\sigma = 0.25 + 0.025t$.
For the first stage, we use the deterministic demand values of 
$\dpar_{1,j=(\ell,w)} = \dpar_{0,\ell} \Exp[\xipar_1^g] - \eta_w \Exp[\xipar_1^{w}] = 1.229 d_{0,\ell} - 1.207 \eta_w.$
The units of all demands (and all primal decision variables) are gigawatts.

We assume there are no holding costs and no costs for removing capacity, i.e., we use $c_{ti}^{\uminusvar} =
c_{ti}^\svar = 0$ for all $i \in [\Ipar], t \in [\Tpar]$. We use discounting to determine the other costs, setting 
$\cpar_{ti}^{\uplusvar} = 5 \iota_i / 1.1^t$, $\cpar_{ti,j=(\ell,w)}^{x} = 0.001 c_i \tau_\ell \tau_w / 1.1^t$ and
$\cpar_{t,j=(\ell,w)}^{z} = \tau_\ell \tau_w / 1.1^t$ where the values of the annualized costs $\iota_i$, 
operation costs $c_i$, $\tau_\ell$ and $\tau_w$ values are from \cite{de2013risk} (see Appendix \ref{sec:appA}). All costs are in million of
Euros.
\new{Finally, we assume the maximum installation per stage is a constant $M_{ti}^{\uplusvar} = C$, and derive redundant upper
bounds on $\svar$ and $\uminusvar$, i.e., $M_{ti}^\svar = \sum_{r=1}^t M_{ri}^{\uplusvar}$ \text{and }
$M_{ti}^{\uminusvar} = M_{t-1,i}^\svar$.
In our experiments we consider two different sets of instances defined using $C=50$ and $C=100$.  }

\subsection{Implementation details}
\label{sec:capimpl}

We compare the primal and dual bounds obtained using two-stage and static LDR. 
\new{For the LDR basis functions, for each $t \in [\Tpar]$, we
let $\Kt = 3$, and }
\begin{align*}
\basis_{t1}(\xit) = 1, \quad \basis_{t2}(\xit) =   \prod_{r=2}^t \xipar_r^g, \quad \basis_{t3}(\xit) = \prod_{r=2}^t \xipar_r^w .  
\end{align*}
Because assumptions A3 and A4 do not hold for this problem (the random variables do not have bounded support and
$\Exp(\xiT | \xit)$ is not linear in $\xit$), the reformulation 
approach from \cite{chen2008linear,kuhn2011primal,shapiro2005complexity} used for \new{the} static LDR cannot be
applied to solve \LDRprimal \ \eqref{eq:ldr} \ and \LDRdual \ \eqref{eq:dldr}. We therefore use a sampling strategy to approximately solve these problems.
Specifically, the sample approximations of \LDRprimal \ \eqref{eq:ldr} \ and \LDRdual \ \eqref{eq:dldr} are
identical to \LDRprimal \ \eqref{eq:ldr} \ and \LDRdual \ \eqref{eq:dldr}, respectively, except that the infinite set of constraints $\allxi$ are replaced
by the finite set corresponding to the sample. 
Approximate solutions to \PartialLDRprimal\ \eqref{eq:ldr2s} and \PartialLDRdual \ \eqref{eq:dldr2s} are obtained by solving the SAA problems
\eqref{prim:saa} and \eqref{dual:saa}, respectively. We solve all sample approximations using the same sample of size $N={\cred 150}
\Tpar$. 

Models \PartialLDRprimal\ \eqref{eq:ldr2s} and \PartialLDRdual\ \eqref{eq:dldr2s} corresponding to model \eqref{eqs:CapExpPrimal} and its dual
are given in Appendix \ref{sec:appB}.
Although the \MSLP \ given in \eqref{eqs:CapExpPrimal} has relatively complete recourse (Assumption \ref{assmp:MSLPrecourse}), 
its two-stage primal approximation \PartialLDRprimal \  \eqref{eq:ldr2s} does not have relatively complete recourse. 
\new{Thus, for obtaining a primal policy using \PartialLDRprimal\ \eqref{eq:ldr2s}, we implement the STT policy proposed in
Section \ref{sec:ubest}. The parameter $\rho$ is determined by conducting a golden section search using a fixed evaluation sample
of $100$ scenarios. Specifically, starting with a lower bound of $0$ and an upper bound of $1000$, a golden section
search is performed in which a simulation of the STT policy with these $100$ scenarios is used to guide the search. In
case the current upper estimate of $\rho$ in the search process yields the minimum estimated cost, the search is
restarted with the new lower estimate set to the current upper estimate of $\rho$, and the new upper estimate set to 
four times the current upper estimate. The search is terminated when either the upper estimate and lower estimates of
$\rho$ differ by less than $1.0$, or the difference in the estimated objective values between the upper and lower
estimates is less than $10^{-6}$ times the sum of the two objective estimates. The resulting value of $\rho$ is then
used in the simulation with $N'=5000\Tpar$ replications to estimate the quality of the resulting policy. The time to select $\rho$ in this process 
was vastly dominated by the time to simulate the policy, but is included in all numerical results that follow.}


Since Assumption
\ref{assmp:Boundedness} is satisfied, any solution to the SAA problem \eqref{dual:saa} provides a feasible solution to
\PartialLDRdual \ \eqref{eq:dldr2s}, and hence evaluating this solution using $N'$ independent replications yields a statistical lower bound
on $\zopt$. \new{In our experiments we use $N'=5000T$ scenarios for estimating the value of this policy.}
Unfortunately, the sample approximations of \LDRprimal\ \eqref{eq:ldr} and \LDRdual \ \eqref{eq:dldr} do not yield policies (primal or
dual) that are feasible under all scenarios. Thus, when evaluating these policies with the independent replications, we report two measures: the average
objective value over scenarios that are feasible, and the fraction of scenarios that are infeasible.  By averaging only over
feasible scenarios, these estimates are optimistically biased, i.e., they underestimate the bound on the primal
problem, and overestimate the bound for the
dual problem. As a result, these estimates do not necessarily provide valid (statistical) upper and lower bounds on
$\zopt$, but we use them to provide a ``best case'' estimate when comparing to the estimates obtained from the two-stage
LDR policies.

All of our numerical results are carried out using IBM ILOG CPLEX 12.6 as the LP solver. {\cred We perform all experiments using a single thread on a Mac OS X 10.12 with 4 GHz Intel Core i7 CPUs and 16 GB RAM.}

 
The SAA problems \eqref{prim:saa} and \eqref{dual:saa} are solved with a sample size of {\cred $N=150T$}.
{\cred The primal SAA problem \eqref{prim:saa} is solved with Benders decomposition, using a single aggregate cut per
time-stage. The Benders decomposition is run until no violated cuts are found. The dual SAA  
problem \eqref{dual:saa} is solved with the bundle-level method \cite{fabian2007,lnnbundle}. The level method for
solving the dual SAA problem is terminated when the relative gap between the lower and upper bounds is less than $10^{-5}$. The details of the Benders decomposition and the level method are provided in Appendix \ref{subsec:appB1} and  \ref{subsec:appB2}, respectively. 
}

\subsection{Comparison between static and two-stage LDR}
\label{subsubsec:PrimalResults}





\new{Tables  \ref{tab:PrimalResults}  and \ref{tab:DualResults} present 95\% confidence intervals (CIs) on the expected costs
of primal policies and dual lower bounds, respectively, obtained using the two-stage and static LDR policies. These
results are reported only for the instances having $T =5,10$. 
The costs are normalized such that for each instance, the estimated lower bound obtained by the two-stage LDR
policy has value 100.0. In these tables, the CIs are presented with their mean and half-width ($\pm$).
In Table \ref{tab:PrimalResults}, the upper end of the CI the two-stage LDR policy is an upper bound on the expected
cost of using that policy, and hence is a statistical upper bound on $\zopt$.  
We also report  under column `Inf.~(\%)' the percentage  of the scenarios (out of $5000 \Tpar$
evaluated scenarios) for which the static LDR policy is infeasible. 
To give an idea of the relative improvement in the expected policy cost obtained by using the two-stage LDR, the column
`$\%$U$\Delta$', presents the percentage increase in the upper bound on the cost obtained with the static policy over the
upper bound on the cost obtained with the two-stage LDR policy.
We observe that the expected cost of the static LDR policy is between 36\% and 102\% higher than that
of the static LDR policy, with the most significant differences occurring with larger time stages.  We also observe
that the static LDR policy is frequently infeasible. Finally, although not presented in the table, we find that the
estimated expected cost of the STT policies was consistently similar (within 2.3\%) to the objective value of the SAA
problem \eqref{prim:saa}, indicating that the STT policy is effectively ``tracking'' the obtained two-stage LDR policy. }

\begin{table}[htbp]
  \centering
    \begin{tabular}{llrrrrrrrr}
    \toprule
          &       &       & \multicolumn{2}{c}{2S LDR} &       & \multicolumn{4}{c}{Static LDR} \\
\cmidrule{4-5}\cmidrule{7-10}    $C$ & $T$     &       & \multicolumn{1}{c}{Mean} & \multicolumn{1}{c}{$\pm$} &       &
\multicolumn{1}{c}{Mean} & \multicolumn{1}{c}{$\pm$} & \multicolumn{1}{c}{Inf. (\%)} &
\multicolumn{1}{c}{\%U$\Delta$ } \\
    \midrule
    50    & 5     &       & 100.8 & 0.3   &       & 138.0 & 0.2   & 3.0   & 36.7 \\
          & 10    &       & 114.6 & 0.6   &       & 232.9 & 0.4   & 3.8   & 102.7 \\*[0.15cm]
    100   & 5     &       & 101.3 & 0.3   &       & 138.4 & 0.2   & 2.9   & 36.5 \\
          & 10    &       & 109.2 & 0.4   &       & 195.8 & 0.3   & 4.0   & 78.8 \\
    \bottomrule
    \end{tabular}%
\caption{Confidence intervals for expected costs of the primal policies.}
  \label{tab:PrimalResults}%
\end{table}%


\new{Considering the CIs of the lower bounds obtained from using two-stage and static LDR policies presented in 
Table \ref{tab:DualResults}, we again find that the static LDR
policy is often infeasible. Column `$\%$L$\Delta$' presents the percentage difference between the lower end of the CI
obtained from the static and two-stage LDR policies, and indicates that the (95\% confidence) lower bounds obtained by
the static LDR range from being similar to 2.9\% lower than those obtained by the two-stage LDR policy.}

\begin{table}[htbp]
  \centering
    \begin{tabular}{llrrrrrrrr}
    \toprule
          &       &       & \multicolumn{2}{c}{2S LDR} &       & \multicolumn{4}{c}{Static LDR} \\
\cmidrule{4-5} \cmidrule{7-10}   $C$ & $T$     &       & \multicolumn{1}{c}{Mean} & \multicolumn{1}{c}{$\pm$} &       &
\multicolumn{1}{c}{Mean} & \multicolumn{1}{c}{$\pm$} & \multicolumn{1}{c}{Inf. (\%)} & \multicolumn{1}{c}{\%
L$\Delta$} \\
    \midrule
    50    & 5     &       & 100.0 & 0.3   &       & 98.7  & 0.3   & 2.4   & -1.3 \\
          & 10    &       & 100.0 & 0.5   &       & 97.1  & 0.4   & 3.5   & -2.9 \\*[0.15cm]
    100   & 5     &       & 100.0 & 0.3   &       & 100.0 & 0.3   & 2.1   & 0.0 \\
          & 10    &       & 100.0 & 0.4   &       & 98.1  & 0.4   & 3.2   & -1.8 \\
    \bottomrule
    \end{tabular}%
\caption{Confidence intervals for expected costs of the dual policies.}
\label{tab:DualResults}
\end{table}%

\subsection{Comparison with SDDP}

\new{
We next compare the two-stage LDR approximation with the results obtained using SDDP.  
In order to apply SDDP, we need a formulation
having a finite number of scenarios per stage and  stage-wise independent random
variables. We obtain a model with stage-wise independence by introducing new state variables $v_{t}^g$ and $v_t^w$ to represent
$\prod_{r=2}^t \xi_r^g$ and $\prod_{r=2}^t \xi_r^w$, respectively, which is implemented by adding the state equations
\begin{equation}
\label{sddp:steq}
v_t^g = \xi_t^g  v_{t-1}^g, \quad v_t^w = \xi_{t}^w v_{t-1}^w, \qquad t \in [2,T] 
\end{equation}
and $v_1^g = v_1^w = 1$. With these state variables, the demand in stage $t \geq 2$ is then represented as
$\max \{\dpar_{0,\ell} v_t^g - \eta_w K_t^{w} v_t^w, 0\}$. In particular, the right-hand side of constraints
\eqref{eq:CapExpPrimal_con3} are replaced with the expression $\dpar_{0,\ell} v_t^g - \eta_w K_t^{w} v_t^w$, and the
redundant upper bounds $\zvar_{tj}(\xit) \leq d_{tj}(\xit)$ in \eqref{eq:CapExpPrimal_con4} are removed.
Thus 
the only random variables appearing in stage $t$  constraints are
$\xi_t^g$ and $\xi_t^w$, which are stage-wise independent. 
We use SAA to construct scenario trees with a finite number of outcomes per stage.
In an SAA problem, we approximate the joint distribution of $\xi_t^g$ and $\xi_t^w$ with $200$ scenarios, obtained
by independent Monte Carlo sampling. Note that 
the SAA approximation has $200^{T-1}$ total sample paths. The number of scenarios per stage was determined based on initial
experiments solving multiple replications of the SAA problem, and was found to provide a good trade-off between difficulty in
solving each individual SAA problem by SDDP and the variability of the SAA estimates.  
The optimal value of an SAA problem is random because it is defined by a random sample. The expected value of this
optimal value is a lower bound on the
true optimal value \cite{makmortwood}. Thus, by solving multiple SAA problems with independent samples, a confidence interval on the expected
value of  the SAA problem, and hence a lower bound on the true optimal value, can be obtained.   
We thus generate 25
independently generated SAA problems, and for each one we obtain a lower bound by solving it with SDDP for a limited
time. These replication values are then used to construct a confidence interval on the lower bound on $\zopt$.

We use the SDDP implementation {\tt sddp.jl} \cite{dowsonsddp} to solve each SAA problem. This algorithm is implemented in Julia. 
In benchmarks reported in \cite{dowsonsddp}, it was found
that the computation times for {\tt sddp.jl}  were about 30\% higher than those for the C++ code DOASA
\cite{philpottdoasa}, on a test instance for which DOASA was designed for. The code {\tt sddp.jl} does not directly
support having random constraint coefficients, as in \eqref{sddp:steq}. However, the algorithm does support solving a
problem with an underlying state evolving according to a Markov chain, and with parameters in the constraints dependent
on the state of the Markov chain. Thus, we model the stochastic process as a Markov chain having $200$ states
corresponding to the 200 scenarios of joint realizations of $(\xi_t^g,\xi_t^w)$
in each stage $t \in [2,T]$.  The transition probability from each state in stage $t$ to each state in stage $t+1$ is
$1/200$. To limit the risk that the cutting plane models used in the SDDP algorithm grow too large, we set the
parameter ``cut\_selection\_frequency'' to $50$, which means that after every 50 iterations of the SDDP algorithm, 
cuts that are not currently binding are removed. 
Finally, to be consistent with the implementation of the two-stage LDR
approximation, we run {\tt sddp.jl} serially, although we note that both {\tt sddp.jl} and the two-stage LDR approximation have significant potential for
speedup via parallelization.

The time limit for each SDDP replication is set as follows. We let \basetime\ be the total time required to solve the SAA problems \eqref{prim:saa} and \eqref{dual:saa}, and
evaluate the value of the obtained dual policy with an independent sample of size $N'=5000T$.  We run the SDDP
algorithm on each of the 25 SAA replications with two time limits: TL:=$1.5*$\basetime/25  and $10*$TL. The first time
limit is used to approximately match the total time (over all replications) allotted to the SDDP algorithm with 
the time used by the two-stage LDR approach (where the factor $1.5$ is used to compensate for the fact that {\tt sddp.jl} is
implemented in Julia whereas the two-stage LDR approach is implemented in C++). The second time limit is used to
demonstrate the potential of SDDP to obtain improved lower bounds and policies when given more time. 
Estimating the expected cost of the SDDP and STT policies requires a separate simulation of these policies, which has
very similar computational effort for the two policies, and thus this time is excluded from \basetime.

The lower bound results are reported in Table \ref{tab:sddplb},
in which again the objective values are scaled such that the estimated lower bound obtained by the two-stage LDR
algorithm is 100.0. In the table, \basetime\ is rounded to the nearest second. In aggregate, 40\% of this time
is spent solving \eqref{prim:saa}, 48\%  is spent solving \eqref{dual:saa}, and 12\% is spent evaluating the dual bound with the
independent sample. The table also presents the mean and half-width ($\pm$) of the lower bound obtained using
two-stage LDR and
the SDDP algorithm given  time limits TL and $10*$TL. The columns
\%L$\Delta$ present the percentage difference between the lower end of the 95\% CI on the lower bound obtained by the SDDP algorithm and that obtained
by the two-stage LDR algorithm.  
\begin{table}[htbp]
  \centering
    \begin{tabular}{llrrrrrrrrrrr}
    \toprule
          &       &   &    \multicolumn{2}{c}{\PartialLDRdual} &       & \multicolumn{3}{c}{SDDP TL} &       &
			 \multicolumn{3}{c}{SDDP 10X TL} \\
\cmidrule{4-5}\cmidrule{7-9}\cmidrule{11-13}    $C$ & $T$     &   \multicolumn{1}{c}{\basetime}    & \multicolumn{1}{c}{Mean} &
\multicolumn{1}{c}{$\pm$} &       & \multicolumn{1}{c}{Mean} & \multicolumn{1}{c}{$\pm$} &
\multicolumn{1}{c}{\%L$\Delta$} &
& \multicolumn{1}{c}{Mean} & \multicolumn{1}{c}{$\pm$} & \multicolumn{1}{c}{\%L$\Delta$} \\
    \midrule
    50    & 5     &  188 & 100.0 & 0.3   &       & 100.6 & 0.7   & 0.2   &       & 100.8 & 0.7   & 0.4 \\
          & 10    & 866   & 100.0 & 0.5   &       & 101.5 & 1.3   & 0.6   &       & 103.2 & 1.3   & 2.4 \\
          & 15    & 1728 & 100.0 & 1.3   &       & 94.3  & 2.6   & -7.1  &       & 103.9 & 2.9   & 2.3 \\
          & 20    & 2897 & 100.0 & 1.5   &       & 80.8  & 3.4   & -21.4 &       & 94.2  & 3.8   & -8.2 \\*[0.15cm]
    100   & 5     & 171    & 100.0 & 0.3   &       & 101.5 & 0.7   & 1.1   &       & 101.7 & 0.7   & 1.3 \\
          & 10    & 1094   & 100.0 & 0.4   &       & 100.8 & 1.1   & 0.1   &       & 101.6 & 1.2   & 0.9 \\
          & 15    & 2235 & 100.0 & 0.9   &       & 102.4 & 2.3   & 1.0   &       & 108.6 & 2.5   & 7.1 \\
          & 20    & 3827 & 100.0 & 1.8   &       & 88.3  & 3.4   & -13.5 &       & 101.4 & 3.8   & -0.6 \\
    \bottomrule        \end{tabular}%
  \caption{Comparison of lower bounds obtained from two-stage LDR and SDDP algorithm.}
  \label{tab:sddplb}%
\end{table}%
Here a negative number indicates the lower bound was smaller (worse), and a positive
number indicates an improvement over two-stage LDR. We find that when given a time limit similar to the time used by the two-stage
LDR approximation, the SDDP algorithm
obtains slightly better lower bounds on instances with fewer time stages, but somewhat worse lower bounds on the
instances with more time stages. On the other hand, when given more time, the SDDP algorithm is able to achieve
noticeably better lower bounds on instances with the fewer time stages, and closes much of the gap on the instances with
more stages. 

We next compare estimates of the expected cost of policies obtained with the two-stage LDR and SDDP methods. For the
two-stage LDR policy, the policy and estimate of associated upper bound are determined as described in Section
\ref{sec:capimpl}. For the SDDP algorithm, a policy can be obtained by first solving a (single) SAA approximation
problem, and then using the resulting value-function approximation to drive a policy that is then evaluated via forward 
simulation replications using independently generated values of the random variables (i.e., independent from those used
in the SAA approximation). Unfortunately, the ability to run a forward simulation using samples different from those
used to solve the SDDP problem is not supported in {\tt sddp.jl}. To obtain an estimate of the value of the
policy that can be obtained using SDDP, for each of the 25 SAA replications solved by SDDP, we simulated the
resulting policy using the sample distribution used in the SAA problem to estimate the expected cost of that policy. We then
constructed a 95\% confidence interval of the resulting upper bounds, and these are the values reported in Table 
\ref{tab:sddpub}.
\begin{table}[htbp]
  \centering
    \begin{tabular}{llrrrrrrrrrrr}
    \toprule
          &       &      &  \multicolumn{2}{c}{\PartialLDRprimal} &       & \multicolumn{3}{c}{SDDP TL} &       &
			 \multicolumn{3}{c}{SDDP 10X TL} \\
\cmidrule{4-5}\cmidrule{7-9}\cmidrule{11-13}    $C$     & $T$     &     \multicolumn{1}{c}{$t_{\mathrm{EVAL}}$}   & \multicolumn{1}{c}{Mean} &
\multicolumn{1}{c}{$\pm$}    &       &
\multicolumn{1}{c}{Mean} &  \multicolumn{1}{c}{$\pm$}     & \multicolumn{1}{c}{\%U$\Delta$} &       & \multicolumn{1}{c}{Mean} & \multicolumn{1}{c}{$\pm$} & \multicolumn{1}{c}{\%U$\Delta$} \\
    \midrule
    50    & 5     & 53 & 100.8 & 0.3   &       & 100.7 & 0.7   & 0.3   &       & 100.8 & 0.7   & 0.4 \\
          & 10    & 239 & 114.6 & 0.6   &       & 104.3 & 1.4   & -8.2  &       & 104.1 & 1.3   & -8.4 \\
          & 15    & 491 & 121.0 & 1.3   &       & 109.5 & 3.0   & -8.0  &       & 108.4 & 3.0   & -8.8 \\
          & 20    & 932 & 102.9 & 1.4   &       & 100.6 & 4.1   & 0.4   &       & 100.4 & 4.0   & 0.1 \\*[0.15cm]
    100   & 5     & 53 & 101.3 & 0.3   &       & 101.7 & 0.7   & 0.8   &       & 101.7 & 0.7   & 0.7 \\
          & 10    & 238 & 109.2 & 0.4   &       & 102.0 & 1.2   & -5.9  &       & 102.0 & 1.2   & -5.9 \\
          & 15    & 495 	& 133.2 & 1.2   &       & 113.4 & 2.6   & -13.6 &       & 112.4 & 2.6   & -14.4 \\
          & 20    & 900 & 116.6 & 1.6   &       & 109.6 & 4.0   & -3.9  &       & 109.0 & 3.9   & -4.5 \\
    \bottomrule
    \end{tabular}%
  \caption{Comparison of approximate upper bounds obtained from two-stage LDR and SDDP algorithm.}
  \label{tab:sddpub}
\end{table}%
The column `$t_{\mathrm{EVAL}}$' in this table presents the time, in seconds, to estimate the expected cost of the STT policy.
The remaining columns present
the confidence intervals of the estimated upper bounds in format similar to Table \ref{tab:sddplb}.
As we see from the columns $\%$U$\Delta$, the estimated expected cost of the SDDP policies is in many cases significantly
lower than the estimated expected cost of the two-stage LDR policy, suggesting that SDDP  
obtains significantly better primal policies for this problem.

In summary, for this problem, we find that SDDP provides similar, or slightly worse, lower bounds, and significantly
better primal policies, in a comparable amount of time as  the two-stage LDR approximation, and the lower bounds can
be improved by running SDDP for more time.  Thus, for this problem, SDDP is clearly favored over the
two-stage approximation. Thus, LDR approximations (both static and two-stage) may be most useful for problems in which the assumptions required to apply SDDP do not hold. 
For example, 
in a hydropower planning case study presented in \cite{shapiro:ejor13}, the time series of water inflows, $X_t$, was
modeled as $X_t = e^{Y_t}$, where $Y_t$ follows a first order $AR(1)$ autoregressive  time series, making the model
nonlinear in $X_t$, and hence not solvable directly by SDDP.}

\section{Concluding remarks}
\label{sec:conc}
We propose two-stage LDRs, a new approximate solution method for MSLPs. This approach has two advantages over staic LDRs. Due to the
flexibility in the recourse decisions, our method potentially yields better (at least not worse) bounds and policies
than standard static LDR policies. In addition, as our approach is based on sampling and 2SLP, it
works with very mild assumptions and can take advantage of existing literature on methods for approximately solving 2SLP problems. 
\new{We illustrate the new approach on two example problems, an inventory planning problem and a capacity planning problem, which indicate that two-stage LDR  policies have potential
to yield significantly better policies than static LDR policies.}

In future research it will be interesting to test the use of two-stage LDR \new{policies} on more problems, and to investigate if there
are problem classes where two-stage LDR \new{policies} are provably optimal or near-optimal.

In the primal problem, \new{a} two-stage LDR can be directly applied to
multi-stage stochastic {\it mixed integer programs}, 
provided the integrality restrictions are imposed only on the recourse variables. 
Since availability of algorithms for multi-stage stochastic mixed integer programs is very limited, it will be interesting to explore this extension
further, in particular possibly using ideas from \cite{bertsimas2015design} to obtain a decision rule in the case the
state variables also have integrality constraints. 
\\ \\
\noindent
{\small {\bf Acknowledgements}. This work is supported in part by the National Science Foundation under grant CMMI-1634597, and by the U.S. Department of Energy, Office of
Science, Office of Advanced Scientific Computing Research, Applied Mathematics program under contract number
DE-AC02-06CH11357.
\bibliographystyle{abbrv}
\bibliography{partialldr}

\newpage
\appendix
\section{Data for the capacity expansion problem}
\label{sec:appA}
$\phantom{0}$
\begin{table}[h!]
\centering
\begin{tabular}{l rrr}
\toprule
& $i=1$ & $i=2$ & $i=3$ \\
\midrule
$\iota_i$ (k\euro/MW) & 245.8 & 113.9 & 57.8 \\
$c_i$ (\euro/MWh) & 41.9 & 58.9 & 90.8 \\
\bottomrule
\end{tabular}
\caption{Fixed annual cost and operation cost (Table 1 in \cite{de2013risk})}
\label{tab:CostData}
\end{table}

\begin{table}[h!]
\centering
\begin{tabular}{l rrrrrrrr}
\toprule
& $\ell=1$ & $\ell=2$ & $\ell=3$ & $\ell=4$  & $\ell=5$  & $\ell=6$  & $\ell=7$  & $\ell=8$ \\
\midrule
$d_0,\ell$ (GW) & 77.1 & 71.4 & 65.7 & 60.1 & 54.4 & 48.8 & 43.1 & 37.4 \\
$\tau_\ell$ (h) & 68 & 677 & 1585 & 1781 & 1367 & 1688 & 1289 & 305 \\
\bottomrule
\end{tabular}
\caption{Initial demand (Table 2 in \cite{de2013risk})}
\label{tab:DemandData}
\end{table}

\begin{table}[h!]
\centering
\begin{tabular}{l rrrrr}
\toprule
& $w=1$ & $w=2$ & $w=3$ & $w=4$  & $w=5$ \\
\midrule
$\eta_w$ (\%) & 92.9 & 81.8 & 54.9 & 21.2 & 0.0 \\
$\tau_w$ (\%) & 19.8 & 21.78 & 18.2 & 26.7 & 13.5 \\
\bottomrule
\end{tabular}
\caption{Wind regimes (Table 3 in \cite{de2013risk})}
\label{tab:WindData}
\end{table}

\section{Additional models for the capacity expansion example}
\label{sec:appB}

\subsection{\cred Primal model and Benders decomposition}
\label{subsec:appB1}

\PartialLDRprimal \ of the capacity expansion model is obtained by substituting 
$$\svar_{ti}(\xit) = \sum_{k \in [\Kt]} \basis_{tk}(\xit) \betavar_{tki}$$ 
in \eqref{eqs:CapExpPrimal}. Dropping $\xipar^t$ dependences for variables to simplify the notation, we obtain
\begin{subequations}
\begin{alignat}{2} 
	\min \ \ & \sum_{i \in [\Ipar]} \Big( c_{1i}^{\uplusvar} \uplusvar_{1i} + c_{1i}^{\uminusvar} \uminusvar_{1i} + \sum_{j \in [\Jpar]} c_{1ij}^{\xvar} \xvar_{1ij}(\xit) \Big)
	+ \sum_{j \in [\Jpar]} c_{1j}^\zvar \zvar_{1j} \\
	& \hspace*{0.7cm} +  \sum_{t \in [\Tpar]} \sum_{i \in [\Ipar]} c_{ti}^\svar \sum_{k \in [\Kt]} \Exp \big[
	\basis_{tk}(\xit) \big] \betavar_{tki} 
	+ \sum_{t \in [2,\Tpar]} \Exp [ \ffnc_t(\betavar,\xit) ] \nonumber \\
		\text{s.t.} \ \ 
		& \betavar_{11i} - \uplusvar_{1i} + \uminusvar_{1i} = 0, && \hspace*{-0.5cm} i \in [\Ipar], \label{eq:capprimalb} \\
		& \betavar_{11i} - \xvar_{1ij} \geq 0, && \hspace*{-0.5cm} i \in [\Ipar], \ j \in [\Jpar],   \\		
		& \sum_{i \in [\Ipar]} \xvar_{1ij} + \zvar_{1j} \geq d_{1j}, && \hspace*{-0.5cm}  j \in [\Jpar], \\
		& 0 \leq \zvar_{1j} \leq d_{1j}, && \hspace*{-0.5cm} j \in [\Jpar],  \\
		&  0 \leq \uplusvar_{1i} \leq M_{1i}^{u^+},  \ \ 0 \leq \uminusvar_{1i} \leq M_{1i}^{u^-},  \ \ 0 \leq
		\betavar_{11i} \leq M_{1i}^{s}, && \hspace*{-0.5cm} i \in [\Ipar],  \\
		& \xvar_{1ij} \geq 0, && \hspace*{-0.5cm} i \in [\Ipar], \ j \in [\Jpar],   \label{eq:capprimalg}
\end{alignat}
\end{subequations}
where, for $t \in [2,\Tpar]$, $\ffnc_t(\betavar,\xit)$ is defined as the optimal objective value of the following problem:

\begin{subequations}
\label{eqs:capSP}
\begin{alignat}{2}
\min  \ \ & \sum_{i \in [\Ipar]} \Big( c_{ti}^{\uplusvar} \uplusvar_{i} + c_{ti}^{\uminusvar} \uminusvar_{i} + \sum_{j \in [\Jpar]} c_{tij}^{\xvar} \xvar_{ij} \Big)
	+ \sum_{j \in [\Jpar]} c_{tj}^\zvar \zvar_{j} && \label{eqs:capSPobj}  \\*[0.2cm]
	 	\text{s.t.} \ \ 
		& \uplusvar_{i} - \uminusvar_{i} = \sum_{k \in [\Kt]} \basis_{tk}(\xit) \betavar_{tki} 
			- \hspace*{-0.2cm} \sum_{k \in [\Ktmone]} \basis_{tk}(\xitmone) \betavar_{t-1,k,i}, \ \ && i \in [\Ipar],  \label{eq:capSPscon1} \\*[0.15cm]
		& \xvar_{ij} \leq \sum_{k \in [\Kt]} \basis_{tk}(\xit) \betavar_{tki}, &&  i \in [\Ipar], \ j \in [\Jpar],  \label{eq:capSPscon2} \\*[0.1cm]
		& \sum_{i \in [\Ipar]} \xvar_{ij} + \zvar_{j} \geq d_{tj}(\xit), &&  j \in [\Jpar], \label{eq:capSPscon3} \\
		& 0 \leq \zvar_{j} \leq d_{tj}(\xit), && j \in [\Jpar],  \label{eq:capSPscon4} \\*[0.22cm]
		&  0 \leq \uplusvar_{i} \leq M_{ti}^{u^+}, && i \in [\Ipar], \label{eq:capSPscon5}  \\*[0.1cm]
		&  0 \leq \uminusvar_{i} \leq M_{ti}^{u^-}, && i \in [\Ipar], \label{eq:capSPscon6} \\*[0.1cm]
		& \xvar_{ij} \geq 0, && i \in [\Ipar], \ j \in [\Jpar],   \label{eq:capSPscon10} \\*[0.1cm]
		&  {0 \leq M_{ti}^s  - \sum_{k \in [\Kt]} \basis_{tk}(\xit) \betavar_{tki},} && i  \in [\Ipar], \label{eq:capSPscon7} \\*[0.15cm]
		& {0 \leq \sum_{k \in [\Kt]} \basis_{tk}(\xit) \betavar_{tki},} &&  i \in [\Ipar], \label{eq:capSPscon8} 
\end{alignat} 
\end{subequations}

We note that \PartialLDRprimal \ of the capacity expansion model does not have relatively complete recourse since the recourse constraints \eqref{eq:capSPscon7} and \eqref{eq:capSPscon8} might be violated under some scenarios. 

{\cred 
Let $\xiT_n, n \in [N]$, be an independent and identically distributed (i.i.d.) random sample of the random vector 
$\xiT$. We solve the obtained primal SAA problem with Benders decomposition, using a single aggregate cut per time-stage. That is, we have a master problem of the following form:
\begin{subequations}
\begin{alignat}{2}
\min \ \ & \sum_{i \in [\Ipar]} \Big( c_{1i}^{\uplusvar} \uplusvar_{1i} + c_{1i}^{\uminusvar} \uminusvar_{1i} + \sum_{j \in [\Jpar]} c_{1ij}^{\xvar} \xvar_{1ij}(\xit) \Big)
	+ \sum_{j \in [\Jpar]} c_{1j}^\zvar \zvar_{1j} \\
	& \hspace*{0.7cm} +  \sum_{t \in [\Tpar]} \sum_{i \in [\Ipar]} c_{ti}^\svar \sum_{k \in [\Kt]} \Exp \big[
	\basis_{tk}(\xit) \big] \betavar_{tki} 
	+ \sum_{t \in [2,\Tpar]} \eta_t \nonumber \\
		\text{s.t.} \ \ 
		& \eqref{eq:capprimalb}-\eqref{eq:capprimalg}, \\
		& (\eta_t, \betavar_{t11}, \hdots, \betavar_{t \Kt \Ipar} ) \in \mathcal{O}^t , &&\hspace*{-0.9cm} t \in [2,\Tpar] , \label{eq:mpoptcuts} \\
		& (\betavar_{t11}, \hdots, \betavar_{t \Kt \Ipar} ) \in \mathcal{F}^t, && \hspace*{-0.9cm} t \in [2,\Tpar],  \label{eq:mpfeascuts} \\
		& {0 \leq \sum_{k \in [\Kt]} \basis_{tk}(\xit_n) \betavar_{tki} \leq M_{ti}^s,} && \hspace*{-0.9cm} t \in [2,\Tpar], i \in [\Ipar], \ n \in [N], \label{eq:mpscenconsts} \\
		& \eta_t \geq 0, && \hspace*{-0.9cm} t \in [2,\Tpar]. \label{eq:mpetanonneg}
\end{alignat} 
\end{subequations}
{\sloppypar The variable $\eta_t$ represents $\frac{1}{N} \sum_{n
\in [N]}   [ \ffnc_t(\betavar,\xit_n) ]$, that is the expected second-stage cost at period $t \in [2,\Tpar]$. Note that since all the original decision variables are defined to be
nonnegative, and all the cost parameters are assumed to be nonnegative, $\ffnc_t(\betavar,\xit) \geq 0$ for any given
$\beta$, thus \eqref{eq:mpetanonneg} are valid. Constraints \eqref{eq:mpoptcuts} and \eqref{eq:mpfeascuts} correspond to
the set of Benders optimality and feasibility cuts, respectively. As $\betavar$ variables belong to the master problem,
we add constraints \eqref{eq:capSPscon7} and \eqref{eq:capSPscon8} for each scenario in the sample to the master problem
as \eqref{eq:mpscenconsts} which can be seen as an additional set of feasibility cuts.}

The subproblem decomposes not only by scenario but also by stage. For $t \in [2,\Tpar]$ and $n \in [N]$, we have the corresponding subproblem \eqref{eqs:capSPobj}-\eqref{eq:capSPscon10}, denoted by SP($t,n$).

At every iteration of the Benders decomposition algorithm, we solve the master problem, get a candidate $\beta$ solution which is fixed in the subproblems, and solve all the subproblems. For $t \in [2,\Tpar]$, if there is at least one index $n \in [N]$ for which SP($t,n$) is infeasible, then we generate a Benders feasibility cut and add it to the master problem. Otherwise, we generate a Benders optimality cut, but add it to the master problem only if it is violated at the current master problem solution. We repeat this procedure until all the subproblems are feasible and no violated optimality cuts are found.
}

\subsection{\cred Dual model and level method}
\label{subsec:appB2}

Let $\Lvar, \Gvar, \Tvar^+, \Tvar^-, \pivar^{\uplusvar},\pivar^{\uminusvar}, \pivar^{\svar}$ be the dual variables
associated with the constraints \eqref{eq:CapExpPrimal_con1}-\eqref{eq:CapExpPrimal_con7} in \eqref{eqs:CapExpPrimal}, respectively. Then, the dual
of \eqref{eqs:CapExpPrimal} is:
\bsubeq
\label{eqs:CapExpDual}
\begin{alignat}{2} 
\max \ \ & \Exp  \sum_{t \in [\Tpar]} \Big[ \sum_{j \in [\Jpar]} \dpar_{tj}(\xit) \big( \Tvar^+_{tj}(\xit) \, {\cred - } \, \Tvar^-_{tj}(\xit)\big)&&   \nonumber \\ 
      & \ \ \, \, {\cred - } \, \sum_{i \in [\Ipar]} \bigl(M_{ti}^{\uplusvar} \pivar^{\uplusvar}_{ti}(\xit) + 
		 M_{ti}^{\uminusvar}\pivar^{\uminusvar}_{ti}(\xit) && + M_{ti}^{\svar} \pivar^{\svar}_{ti}(\xit)\bigr) \Big] \label{eq:CapExpPDual_obj} \\ 
		\text{s.t.} \ \ 
		& \Lvar_{ti}(\xit) - \Exp [ \Lvar_{t+1,i}(\xitpone) \mid \xit ]  \nonumber \\
		& \ + \sum_{j \in [\Jpar]} \Gvar_{tij}(\xit) \, {\cred - } \, \pivar^{\svar}_{ti}(\xit) \leq \cpar_{ti}^{\svar}, && \! t \in [\Tpar], \ \allxi, \ i \in [\Ipar],  \label{eq:CapExpPDual_con1} \\
		& {\cred - } \, \pivar^{\uplusvar}_{ti}(\xit) \hspace*{0.03cm}  -\hspace*{0.03cm} \Lvar_{ti}(\xit)  \leq \cpar_{ti}^{\uplusvar},
		&& \! t \in [\Tpar], \ \allxi, \ i \in [\Ipar],  \label{eq:CapExpPDual_con2} \\	
		& {\cred - } \, \pivar^{\uminusvar}_{ti}(\xit) \hspace*{0.03cm}  + \hspace*{0.03cm} \Lvar_{ti}(\xit) \leq \cpar_{ti}^{\uminusvar}, &&  \! t \in [\Tpar], \ \allxi, \ i \in [\Ipar],  \label{eq:CapExpPDual_con3} \\
		& \Tvar^+_{tj}(\xit) \hspace*{0.05cm}  \, {\cred - } \, \hspace*{0.05cm} \Tvar^-_{tj}(\xit) \ \leq \cpar_{tj}^{\zvar}, && \! t \in [\Tpar], \ \allxi, \ j \in [\Jpar],  \label{eq:CapExpPDual_con4} \\
		& \Tvar^+_{tj}(\xit) - \Gvar_{tij}(\xit)  \leq \cpar_{tij}^{\xvar}, &&  \! t \in [\Tpar], \ \allxi, \ i \in [\Ipar], \ j \in [\Jpar],  \label{eq:CapExpPDual_con5} \\
		&  \Gvar_{tij}(\xit) \geq 0, && \! t \in [\Tpar], \ \allxi, \ i \in [\Ipar], \ j \in [\Jpar], \label{eq:CapExpPDual_con6} \\
		&  \Tvar^+_{tj}(\xit), {\cred \Tvar^-_{tj}(\xit)} \geq 0,  && \!  t \in [\Tpar], \ \allxi,  \ j \in [\Jpar],  \label{eq:CapExpPDual_con7}  \\
		& \pivar^{\uplusvar}_{ti}(\xit), \pivar^{\uminusvar}_{ti}(\xit),  \pivar^{\svar}_{ti}(\xit) \, {\cred \geq} \, 0 && \! t \in [\Tpar], \ \allxi, \ i \in [\Ipar].  \label{eq:CapExpPDual_con8} 
\end{alignat}
\esubeq

{\cred Observing that $\Tvar^-$ variables are redundant, we  remove them to simplify the model.}
\PartialLDRdual \ of the capacity expansion model is obtained by substituting 
$$\Lvar_{ti}(\xit) = \sum_{k \in [\Kt]} \basis_{tk}(\xit) \alphavar_{tki}$$ 
in \eqref{eqs:CapExpDual}. Dropping $\xipar^t$ dependences for variables to simplify the notation, we obtain
\bsubeq
\begin{align} 
	 \max \ \ & \sum_{j \in [\Jpar]} \dpar_{1j} \Tvar^+_{1j}
	 \, {\cred -} \, \sum_{i \in [\Ipar]} \big( M_{1i}^{\uplusvar} \pivar^{\uplusvar}_{1i} + M_{1i}^{\uminusvar}
	 \pivar^{\uminusvar}_{1i} + M_{1i}^{\svar} \pivar^{\svar}_{1i} \big)
	  + \sum_{t \in [2,\Tpar]} && \hspace*{-1cm} \Exp [\fhatfnc_t(\alphavar,\xit) ] \\ 
		\text{s.t.} \ \ 
		& \alphavar_{11i} - \sum_{k \in [\Kpar_2]} \Exp \big[ \basis_{2k}(\xipar^2) \big] \alphavar_{2ki} + \sum_{j \in [\Jpar]} \Gvar_{1ij}  \, {\cred -} \, \pivar^{\svar}_{1i} \leq \cpar_{1i}^{\svar}, && \hspace*{-1.2cm} i \in [\Ipar],  \label{eq:dualfirstconst} \\
		& {\cred -} \, \pivar^{\uplusvar}_{1i} \hspace*{0.03cm}  -\hspace*{0.03cm} \alphavar_{11i} \leq \cpar_{1i}^{\uplusvar}, && \hspace*{-1.2cm} i \in [\Ipar], \\
		& {\cred -} \, \pivar^{\uminusvar}_{1i} \hspace*{0.03cm}  + \hspace*{0.03cm} \alphavar_{11i} \leq \cpar_{1i}^{\uminusvar}, && \hspace*{-1.2cm} i \in [\Ipar],  \\
		& \Tvar^+_{1j} \hspace*{0.05cm}  \leq \cpar_{1j}^{\zvar}, && \hspace*{-1.2cm} j \in [\Jpar], \\
		& \Tvar^+_{1j} - \Gvar_{1ij}  \leq \cpar_{1ij}^{\xvar}, && \hspace*{-1.2cm} i \in [\Ipar], \ j \in [\Jpar]  \\
		&  \Gvar_{1ij} \geq 0, && \hspace*{-1.2cm} i \in [\Ipar], \ j \in [\Jpar]  \\
		&  \Tvar^+_{1j} \geq 0, && \hspace*{-1.2cm} j \in [\Jpar],  \\
		& \pivar^{\uplusvar}_{1i}, \pivar^{\uminusvar}_{1i},  \pivar^{\svar}_{1i} \ {\cred \geq} \ 0 && \hspace*{-1.2cm}  i \in [\Ipar], \label{eq:duallastconst}
\end{align}
\esubeq
where, for $t \in [2,\Tpar]$, $\fhatfnc_t(\alphavar,\xit)$ is defined as the optimal objective value of the following problem:

\bsubeq
\label{eq:dualsp}
\begin{align} 
	 \max \ \ & \sum_{j \in [\Jpar]} \dpar_{tj}  \Tvar^+_{j} 
	   \, {\cred -} \, \sum_{i \in [\Ipar]} \big(M_{ti}^{\uplusvar}\pivar^{\uplusvar}_{i}+  M_{ti}^{\uminusvar}  \pivar^{\uminusvar}_{i}+M_{ti}^{\svar} \pivar^{\svar}_{i} 
	   \big) \\ 
		\text{s.t.} \ \ 		
		& \sum_{j \in [\Jpar]} \Gvar_{ij} \, {\cred -} \, \pivar^{\svar}_{i} \leq \cpar_{ti}^{\svar} - \sum_{k \in [\Kt]}
		\basis_{tk}(\xit) \alphavar_{tki} \nonumber \\
		&  \hspace*{2.54cm} + \sum_{k \in [\Ktpone]} \Exp \Big[ \basis_{t+1,k}(\xitpone) \Big | \xit \Big] \alphavar_{t+1,k,i}, && \hspace*{-0.3cm} i \in [\Ipar],  \label{eq:dualspCon1} \\
		&  {\cred -} \, \pivar^{\uplusvar}_{i} \hspace*{0.03cm} \leq \cpar_{ti}^{\uplusvar} + \sum_{k \in [\Kt]} \basis_{tk}(\xit)
		\alphavar_{tki}, && \hspace*{-0.2cm} i \in [\Ipar], \label{eq:dualspCon2} \\
		& {\cred -} \,  \pivar^{\uminusvar}_{i} \hspace*{0.03cm} \leq \cpar_{ti}^{\uminusvar} - \sum_{k \in [\Kt]} \basis_{tk}(\xit)
		\alphavar_{tki}, && \hspace*{-0.2cm} i \in [\Ipar], \label{eq:dualspCon3} \\
		& \Tvar^+_{j} \hspace*{0.05cm}  \leq \cpar_{tj}^{\zvar}, && \hspace*{-0.2cm} j \in [\Jpar], \label{eq:dualspCon4} \\
		& \Tvar^+_{j} - \Gvar_{ij}  \leq \cpar_{tij}^{\xvar}, && \hspace*{-0.2cm} i \in [\Ipar], \ j \in [\Jpar], \label{eq:dualspCon5} \\
		&  \Gvar_{ij} \geq 0, && \hspace*{-0.2cm} i \in [\Ipar], \ j \in [\Jpar],  \label{eq:dualspCon6} \\
		&  \Tvar^+_{j} \geq 0,  && \hspace*{-0.2cm} j \in [\Jpar], \label{eq:dualspCon7} \\
		& \pivar^{\uplusvar}_{i}, \pivar^{\uminusvar}_{i},  \pivar^{\svar}_{i} \ {\cred \geq} \ 0 && \hspace*{-0.2cm}  i \in [\Ipar].  \label{eq:dualspCon8}
\end{align}
\esubeq


{\cred 
Let $\xiT_n, n \in [N]$, be an independent and identically distributed (i.i.d.) random sample of the random vector 
$\xiT$. We solve the obtained dual SAA problem with the bundle-level method, because the Benders decomposition method
converged slowly for this problem. We use cuts aggregated over scenarios, thus introduce $\zeta_{t}$ variable to represent the expected second-stage cost value, i.e., $\frac{1}{N} \sum_{n \in [N]} \fhatfnc_t(\alphavar,\xit_n)$, for $t \in [2,\Tpar]$. 

We observe that the subproblem \eqref{eq:dualsp} can be further decomposed into two: one problem including only the $\uplusvar$ and $\uminusvar$ variables, and the other problem including the remaining set of variables.
\begin{align*}
(\text{DSP}^{upart}): \ \max \ \ &  - \sum_{i \in [\Ipar]} \big(M_{ti}^{\uplusvar}\pivar^{\uplusvar}_{i}+  M_{ti}^{\uminusvar}  \pivar^{\uminusvar}_{i}  \big) \\ 
		\text{s.t.} \ \ 		
		& \eqref{eq:dualspCon2},\eqref{eq:dualspCon3},\eqref{eq:dualspCon8}
\end{align*} 

\begin{align*}
(\text{DSP}^{rest}): \ \max \ \ & \sum_{j \in [\Jpar]} \dpar_{tj} \Tvar^+_{j} - \sum_{i \in [\Ipar]} M_{ti}^{\svar} \pivar^{\svar}_{i}  \\ 
		\text{s.t.} \ \ 		
		& \eqref{eq:dualspCon1},\eqref{eq:dualspCon4}-\eqref{eq:dualspCon7}
\end{align*} 
We exploit this decomposition to disaggregate the optimality cuts in the master problem. Thus, we introduce additional variables $\zeta^{upart}_{t}$ and $\zeta^{rest}_{t}$ for $t \in [2,\Tpar]$ and obtain the following master problem:
\begin{subequations}
\begin{alignat}{2}
\text{(MP)}: \ \max \ \ & \sum_{j \in [\Jpar]} \dpar_{1j} \Tvar^+_{1j}
	 \, {\cred -} \, \sum_{i \in [\Ipar]} \big( M_{1i}^{\uplusvar} \pivar^{\uplusvar}_{1i} + M_{1i}^{\uminusvar}
	 \pivar^{\uminusvar}_{1i} + M_{1i}^{\svar} \pivar^{\svar}_{1i} \big)
	&&  + \sum_{t \in [2,\Tpar]} \zeta_{t}  \\ 
		\text{s.t.} \ \ 
		& \eqref{eq:dualfirstconst}-\eqref{eq:duallastconst}, \label{eq:levelmpfirstconst}\\
		& \zeta_{t} = \zeta^{upart}_{t} + \zeta^{rest}_{t}, &&\hspace*{-0.9cm} t \in [2,\Tpar], \\
		& (\zeta^{upart}_{t}, \alphavar_{t11}, \hdots, \alphavar_{t \Kt \Ipar} ) \in \mathcal{U}^{t} , &&\hspace*{-0.9cm} t \in [2,\Tpar], \\
		& (\zeta^{rest}_{t}, \alphavar_{t11}, \hdots, \alphavar_{t \Kt \Ipar} ) \in \mathcal{R}^{t} , &&\hspace*{-0.9cm} t \in [2,\Tpar], \\
		& \zeta^{upart}_{t} \leq 0, &&\hspace*{-0.9cm} t \in [2,\Tpar], \\
		& \zeta^{rest}_{t} \leq \frac{1}{N} \sum_{n \in [N]} \sum_{j \in [\Jpar]} \dpar_{tj}(\xit_n) \cpar_{tj}^{\zvar}, &&\hspace*{-0.9cm} t \in [2,\Tpar], \label{eq:levelmplastconst}
\end{alignat} 
\end{subequations}
where $\mathcal{U}^{t}$ and $\mathcal{R}^{t}$ represent the optimality cuts derived from problems ($\text{DSP}^{upart}$)
and ($\text{DSP}^{rest}$), respectively. Moreover, we introduce the upper bounds on the new auxiliary variables, which are derived from the subproblems $(\text{DSP}^{upart})$ and $(\text{DSP}^{rest})$.

The level method also uses a quadratic program for regularization which projects the previous iterate on the level set of the current approximation of the objective function. We use the following problem for this projection:
\begin{alignat*}{2}
\text{(QP)}: \ \max \ \ & || \alphavar - \hat{\alphavar} ||^2_2 \\ 
		\text{s.t.} \ \ 
		& \eqref{eq:levelmpfirstconst}-\eqref{eq:levelmplastconst}, \\
		& \sum_{j \in [\Jpar]} \dpar_{1j} \Tvar^+_{1j}
	 \, {\cred -} \, \sum_{i \in [\Ipar]} \big( M_{1i}^{\uplusvar} \pivar^{\uplusvar}_{1i} + M_{1i}^{\uminusvar}
	 \pivar^{\uminusvar}_{1i} + M_{1i}^{\svar} \pivar^{\svar}_{1i} \big) \geq L,
\end{alignat*} 
where $\hat \alphavar$ and $L$ denote the current $\alphavar$ solution (i.e., the previous iterate) and the level target, respectively. The optimal solution values of $\alphavar$ variables determine the next iterate. 

The details of the level method are provided in Algorithm \ref{algo:LevelAlgo} where LB and UB denote lower bound and upper bound, respectively.

\begin{center} 
\begin{algorithm}[H]
\begin{tabbing}
	\= \qquad \= \qquad \= \qquad \= \qquad \=  \kill
	\> 1. Initialize $\hat{\alphavar} = 0$, $\text{LB} = -\infty$, $\text{UB} = \infty$ \\*[0.2cm]
	\> 2. Solve all the subproblems, i.e., $(\text{DSP}^{upart})$ and $(\text{DSP}^{rest})$ for all $t \in [2,\Tpar]$ and
	$n \in [N]$. \\
	 \hspace*{0.27cm} Generate Benders optimality cuts, add them to both (MP) and (QP).  \\
	 \hspace*{0.27cm} Compute the objective value of the current iterate, and set it as LB. \\*[0.2cm]
	 \> 3. do \\*[0.1cm]
	 \> \> Solve (MP). Update UB if (MP) optimal objective value is lower than UB. \\
	 \> \> Set $L = 0.3 \times \text{UB} + 0.7 \times \text{LB}$. \\
	 \> \> Solve (QP) with updated level constraint, to obtain iterate $\hat{\alphavar}$. \\
	 \> \> Solve all the subproblems at current iterate. \\
	 \> \> Generate Benders optimality cuts, and add violated cuts to both (MP) and (QP). \\
	 \> \> Compute the objective value of the current iterate; if it is larger than LB, update LB.  \\*[0.1cm]
	 \> \hspace*{0.24cm}  until $|$UB - LB$ | \ / $ UB $ < 10^{-5}$
	\end{tabbing}
\caption{: Level Algorithm}
\label{algo:LevelAlgo}
\end{algorithm}
\end{center} 
}

\end{document}